\documentclass[submission,copyright,creativecommons]{eptcs}
\providecommand{\event}{ACT 2019}
\usepackage{underscore}

\usepackage{amsmath}
\usepackage{amssymb}
\usepackage{amsthm}
\usepackage[utf8]{inputenc}
\usepackage[english]{babel}
\usepackage{mathtools}
\usepackage{pict2e}
\usepackage[mathscr]{euscript}
\usepackage{tikz-cd}
\tikzstyle{none}=[inner sep=0pt]
\tikzstyle{circ}=[circle,fill=black,draw,inner sep=3pt]
\usepackage{fancybox}
\usepackage[all,2cell]{xy} \UseAllTwocells
\usetikzlibrary{arrows, positioning, intersections}
\tikzset{%
	symbol/.style={%
		draw=none,
		every to/.append style={%
			edge node={node [sloped, allow upside down, auto=false]{$#1$}}}
	}
}

\newcommand{\define}[1]{{\bf \boldmath{#1}}}

\theoremstyle{definition}
\newtheorem{theorem}{Theorem}
\newtheorem{definition}[theorem]{Definition}
\newtheorem{lemma}[theorem]{Lemma}
\newtheorem{example}[theorem]{Example}

\def\ld{\rotatebox[origin=c]{-90}{$\dashv$}} 

\newcommand{\sSet}{\mathsf{sSet}}
\newcommand{\Th}{\mathsf{Th}}
\newcommand{\RGph}{\mathsf{RGph}}
\newcommand{\Gph}{\mathsf{Gph}}
\newcommand{\Set}{\mathsf{Set}}

\newcommand{\Cart}{\mathsf{Cart}}
\newcommand{\Cat}{\mathsf{Cat}}
\newcommand{\Law}{\mathsf{Law}}

\newcommand{\Top}{\mathsf{Top}}
\newcommand{\Mon}{\mathsf{Mon}}

\newcommand{\Pos}{\mathsf{Pos}}
\newcommand{\Mod}{\mathsf{Mod}}
\newcommand{\Enr}{\mathsf{Enr}}
\newcommand{\Con}{\mathsf{Con}}
\newcommand{\FinSet}{\mathsf{FinSet}}
\newcommand{\NN}{\mathsf{N}}
\newcommand{\A}{\mathsf{A}}
\newcommand{\V}{\mathsf{V}}
\newcommand{\W}{\mathsf{W}}
\newcommand{\D}{\mathsf{D}}
\newcommand{\C}{\mathsf{C}}
\newcommand{\G}{\mathsf{G}}
\newcommand{\R}{\mathsf{R}}
\newcommand{\X}{\mathsf{X}}

\newcommand{\J}{\mathsf{J}}
\newcommand{\T}{\mathsf{T}}
\newcommand{\Kl}{\mathsf{Kl}}

\newcommand{\SKI}{\mathsf{SKI}}

\newcommand{\F}{\mathrm{F}}
\newcommand{\FC}{\mathrm{FC}}
\newcommand{\FP}{\mathrm{FP}}
\newcommand{\FS}{\mathrm{FS}}
\newcommand{\UC}{\mathrm{UC}}
\newcommand{\UP}{\mathrm{UP}}

\newcommand{\US}{\mathrm{UsS}}

\newcommand{\op}{\mathrm{op}}
\newcommand{\Obj}{\mathrm{Ob}}

\newcommand{\maps}{\colon}
\newcommand{\id}{\mathrm{id}}

\title{Enriched Lawvere Theories for Operational Semantics}
\author{John C.\ Baez and Christian Williams
  \\ Department of Mathematics \\
  U.\ C.\ Riverside \\
  Riverside, CA \\
  92521 USA
\email{baez@math.ucr.edu,  cwill041@ucr.edu}
}

\def\titlerunning{Enrichment for Operational Semantics}
\def\authorrunning{John\ C.\ Baez \& Christian Williams}
\begin{document}
\maketitle

\begin{abstract} 
Enriched Lawvere theories are a generalization of Lawvere theories that allow us to describe the operational semantics of formal systems.  For example, a graph-enriched Lawvere theory describes structures that have a \emph{graph} of operations of each arity, where the vertices are operations and the edges are \emph{rewrites} between operations. Enriched theories can be used to equip systems with operational semantics, and maps between enriching categories can serve to translate between different forms of operational and denotational semantics.  The Grothendieck construction lets us study all models of all enriched theories in all contexts in a single category.  We illustrate these ideas with the $SKI$-combinator calculus, a variable-free version of the lambda calculus.
\end{abstract}

\section{Introduction}
\label{sec:intro}

Formal systems are not always explicitly connected to how they operate in practice.   Lawvere theories  \cite{lawvere} are an excellent formalism for describing algebraic structures obeying equational laws, but they do not specify how to compute in such a structure, for example taking a complex expression and simplifying it using rewrite rules.   Recall that a Lawvere theory is a category with finite products $\T$ generated by a single object $t$, for ``type'', and morphisms $t^n \to t$ representing $n$-ary operations, with commutative diagrams specifying equations.   There is a theory for groups, a theory for rings, and so on.   We can specify algebraic structures of a given kind in some category $\C$ with finite products by a product-preserving functor $\mu \maps\T \to \C$.   This is a simple and elegant form of \emph{denotational} semantics.    However, Lawvere theories know nothing of \emph{operational} semantics.  Our goal here is to address this using ``enriched'' Lawvere theories.

In a Lawvere theory, the objects are types and the morphisms are terms; however, there are no relations between terms, only equations. The process of computing one term into another should be given by hom-objects with more structure.  In operational semantics, program behavior is often specified by labelled transition systems, or labelled directed multigraphs \cite{sos}.  The edges of such a graph represent rewrites:
\begin{center}\begin{tikzcd}(\lambda x.x+x \; \; 2) \ar{r}{\beta} & 2+2 \ar{r}{+} & 4\end{tikzcd}\end{center}
We can use an enhanced Lawvere theory in which, rather than merely \emph{sets} of morphisms, there are \emph{graphs} or perhaps \emph{categories}. Enriched Lawvere theories are exactly for this purpose.

In a theory $\T$ enriched in a category $\V$ of some kind of ``directed object'', including graphs, categories, and posets, the theory has the following interpretation:

\[\begin{array}{rl}
\text{types: } & \text{objects of } \T\\
\text{terms: } & \text{morphisms of } \T\\
\text{equations between terms: } & \text{commuting diagrams}\\
\text{rewrites between terms: } & \text{``edges'' in hom in } \V\\
\end{array}\]

To be clear, this is not a new idea. Using enriched Lawvere theories for operational semantics has been explored in the past. For example, category-enriched theories have been studied by Seely \cite{seely} for the $\lambda$-calculus, and poset-enriched ones by Ghani and L\"uth \cite{ghani} for understanding ``modularity'' in term rewriting systems.  They have been utilized extensively by Power, enriching in $\omega$-complete partial orders to study recursion \cite{compeffects} -- in fact, there the simplified ``natural number'' enriched theories which we explore were implicitly considered.

The goal of this paper is to give a simple unified explanation of enriched Lawvere theories and some of their applications to operational semantics.    We aim our explanations at readers familiar with category theory but not yet enriched categories.  To reduce the technical overhead we only consider enrichment over cartesian closed categories.

In general for a cartesian closed category $\V$, a \textbf{$\V$-theory} is a $\V$-enriched Lawvere theory with natural number arities. We consider $\V$ as a choice of ``method of computation'' for algebraic theories.  The main idea of this paper is that product-preserving functors between enriching categories allow for the translation between different kinds of semantics.   This translation could be called ``change of computation''---or, following standard mathematical terminology, \textbf{change of base}.

Because operational semantics uses graphs to represent terms and rewrites, one might expect some category like $\Gph$, the category of directed multigraphs, to be our main example of enriching category: that is, the ``thing'' of $n$-ary operations, or $n$-variable terms in a theory, is a directed graph whose edges are rewrites. This is known as \textit{small-step} operational semantics, meaning each edge represents a single instance of a rewrite rule.

When studying formal languages, one wants to pass from this local view to a global view: given a term, one cares about its possible evolutions after not only one rewrite but any finite sequence of rewrites. We study how programs operate in finite time.  In computer science, this corresponds to defining a rewrite relation and forming its transitive closure, called \textit{big-step} operational semantics. This is the classic example which change of base aims to generalize.

However, there is a subtlety.  We may try to model the translation from small-step to big-step operational semantics using the ``free category'' functor, which for any directed multigraph forms the category whose objects are vertices and morphisms are finite paths of edges.   However, this functor does \emph{not} preserve products.   One might hope to cure this using a better-behaved variant of directed multigraphs, such as reflexive graphs.    One advantage of reflexive graphs is that that each vertex has a distinguished edge from it to itself; these describe rewrites that ``do nothing''.  Thus, in a product of reflexive graphs there are edges describing the process of rewriting one factor while doing nothing in the other.  This lets us handle parallelism.  Unfortunately, as we shall explain, the free category functor from reflexive graphs to categories still fails to preserve products.

To obtain a product-preserving change of base taking us from small-step to big-step operational semantics, it seems the cleanest solution is to generalize graphs to \emph{simplicial sets}.  A simplicial set is a contravariant functor from the category $\Delta$ of finite linear orders and monotone maps to the category of sets and functions.  It can be visualized as a space built from ``simplices'', which generalize triangles to any dimension: point, line, triangle, tetrahedron, etc. For an introduction to simplicial sets, see Friedman \cite{sset}.  We use $\sSet$ to denote the category of simplicial sets, namely $\Set^{\Delta^{\op}}$.

Simplicial sets allow one to generalize rewriting to \emph{higher-dimensional} rewriting, but this is not our focus here.  Indeed, we only need two facts about simplicial sets in this paper:

\begin{itemize}
\item There is a full and faithful embedding of $\RGph$, the category of reflexive graphs, in $\sSet$, so we can think of a reflexive graph as a special kind of simplicial set (namely one whose $n$-simplices for $n > 1$ are all degenerate).
\item The free category functor $\FC \maps \sSet \to \Cat$, often called ``realization'', preserves products.
\end{itemize}

We thus obtain a spectrum of cartesian closed categories $\V$ to enrich over, each connected to the next by a product-preserving functor, which allow us to examine the computation of term calculi in various ways: \\
\[\begin{array}{lll}
\textbf{Simplicial Sets} & \sSet \text{-theories represent ``small-step'' operational semantics:} \\ & \text{--- an edge is a \textit{single} term rewrite.}\\
\textbf{Categories} & \Cat\text{-theories represent ``big-step'' operational semantics:} \\  & \text{(Often this means a rewrite to a normal form. We use the term more generally.)}\\ & \text{--- a morphism is a \textit{finite sequence} of rewrites.}\\
    \textbf{Posets} & \Pos \text{-theories represent ``full-step'' operational semantics:}\\ & \text{--- a boolean is the \textit{existence} of a big-step rewrite.}\\
\textbf{Sets} & \Set \text{-theories represent denotational semantics:} \\ & \text{--- an element is a \textit{connected component} of the rewrite relation.} 
\end{array}\]

In Section \ref{sec:lawvere} we review Lawvere theories as a more explicit, but equivalent, presentation of finitary monads. In Section \ref{sec:enrichment}, we recall the basics of enrichment over cartesian closed categories.   In Section \ref{sec:enriched_lawvere} we give the central definition of $\V$-theory, adapted from the work of Lucyshyn-Wright \cite{lucyshyn-wright}.  Using his work we show that a $\V$-theory $\T$ gives a monadic adjunction between $\V$ and the $\V$-category of models of $\T$ in $\V$.  This generalizes a fundamental result for Lawvere theories.

In Section \ref{sec:base_change} we discuss how suitable functors between enriching categories induce \textit{change of base}: they transform theories, and their models, from one method of rewriting to another.   Our main examples arise from this chain of adjunctions:
\[\begin{tikzcd}[column sep=small]
\sSet  \arrow[bend left,below]{rr}{\FC}
& \ld &
\arrow[bend left,above]{ll}{\US} 
\Cat \arrow[bend left,below]{rr}{\FP}
& \ld &
\arrow[bend left,above]{ll}{\UC} \Pos \arrow[bend left,below]{rr}{\FS}
& \ld &
\Set \arrow[bend left,above]{ll}{\UP}
\end{tikzcd}\]
The right adjoints here automatically preserve finite products, but the left adjoints do as well, and
these are what we really need:

\begin{itemize}
\item
The functor $\FC \maps \sSet \to \Cat$ maps a simplicial set (for example a reflexive graph) to the category it freely generates.  Change of base along $\FC$ maps small-step operational semantics to big-step operational semantics.
\item 
The functor $\FP \maps \Cat \to \Pos$ maps a category $C$ to the poset whose elements are objects of $C$, with $c \le c'$ iff $C$ has a morphism from $c$ to $c'$.
Change of base along $\FP$ maps big-step operational semantics to full-step operational semantics.
\item 
The functor $\FS \maps \Pos \to \Set$ maps a poset $P$ to the set of ``components'' of $P$, where $p,p' \in P$ are in the same component if $p \le p'$. Change of base along $\FS$ maps full-step operational semantics to denotational semantics.
\end{itemize}

In Section \ref{sec:V-theories} we show that models of all $\V$-theories for all enriching $\V$ can be assimilated into one category using the Grothendieck construction.  In Section \ref{sec:applications} we bring all the strands together and demonstrate these concepts in applications.  First we consider the $SKI$-combinator calculus, and then we show how theories enriched over the category of labelled graphs can be used to study bisimulation.

\subsection*{Acknowledgements}

This paper builds upon the ideas of Mike Stay and Greg Meredith presented in ``Representing operational semantics with enriched Lawvere theories''  \cite{roswelt}.  We appreciate their offer to let us develop this work further for use in the innovative distributed computing system RChain, and gratefully acknowledge the support of Pyrofex Corporation.  We also thank Richard Garner, Todd Trimble and others at the $n$-Category Caf\'e for classifying cartesian closed categories where every object is a finite coproduct of copies of the terminal object \cite{nCafe}.

\section{Lawvere Theories}
\label{sec:lawvere}

Algebraic structures are traditionally treated as sets equipped with operations obeying equations, but we can generalize such structures to live in any category with finite products.  For example,
given any category $\C$ with finite products, we can define a monoid internal to $\C$ to consist of:
\[\begin{array}{rl}
\text{an object} & M\\
\text{an identity element} & e\maps 1 \to M\\
\text{and multiplication} & m\maps M^2 \to M\\
\text{obeying the associative law} & m \circ (m \times M) = m \circ (M \times m)\\
\text{and the right and left unit laws} & m \circ (e  \times \id_M) = \id_M = m \circ (\id_M \times e).\\
\end{array}\]
Lawvere theories formalize this idea.  For example, there is a Lawvere theory $\Th(\Mon)$, the category with finite products freely generated by an object $t$ equipped with an identity element $e \maps 1 \to t$ and multiplication $m \maps t^2 \to t$ obeying the associative law and unit laws listed above.    This captures the ``Platonic idea'' of a monoid internal to a category with finite products.  A monoid internal to $\C$ then corresponds to a functor $\mu \maps \T \to \C$ that preserves finite products.  

In more detail, let $\NN$ be any skeleton of the category of finite sets $\FinSet$.  Because $\NN$ is the free category with finite coproducts on $1$, $\NN^\op$ is the free category with finite products on $1$.   A \define{Lawvere theory} is a category with finite products $\T$ equipped with a functor $\tau \maps \NN^\op \to \T$ that is the identity on objects and preserves finite products.   Thus, a Lawvere theory is essentially a category generated by one object $\tau(1) = t$ and $n$-ary operations $t^n \to t$, as well as the projection and diagonal morphisms of finite products.

For efficiency let us call a functor that preserves finite products \define{cartesian}.   Lawvere theories are the objects of a category $\Law$ whose morphisms are cartesian functors $f \maps \T\to \T'$ that obey $f\tau = \tau'$.   More generally, for any category with finite products $\C$, a \define{model} of the Lawvere theory $\T$ in $\C$ is a cartesian functor $\mu \maps \T \to \C$.  The models of $\T$ in $\C$ are the objects of a category $\Mod(\T,\C)$, in which the morphisms are natural transformations.  

A theory can thus have models in many different contexts.   For example, there is a Lawvere theory $\Th(\Mon)$, the theory of monoids, described as above.  Ordinary monoids are models of this theory in $\Set$, while topological monoids are models of this theory in $\Top$.

For completeness, it is worthwhile to mention the \textit{presentation} of a Lawvere theory: how exactly does the above ``sketch'' of $\Th(\Mon)$ produce a category with finite products?  It is precisely analogous to the presentation of an algebra by generators and relations: we form the free category with finite products on the data given, and impose the required equations. The result is a category whose objects are powers of $t$, and whose morphisms are composites of products of the morphisms in $\Th(\Mon)$, projections, deletions, symmetries and diagonals.  A detailed account was given by Barr and Wells \cite[Chap.\ 4]{barrwells}.

In 1965, Linton \cite{linton} proved that Lawvere theories correspond to ``finitary monads'' on the category of sets.   For every Lawvere theory
$\T$, there is an adjunction:
\[\begin{tikzcd}[column sep=small]
\Set \arrow[bend left,below]{rr}{F}
& {\phantom{AA} \ld} &
\arrow[bend left,above]{ll}{U} \Mod(\T,\Set).
\end{tikzcd}\]
The functor 
\[  U \maps \Mod(\T,\Set) \to \Set \]
sends each model $\mu$ to its underlying set, $X = \mu(\tau(1))$. 
Its left adjoint, the free model functor 
\[       F \maps\Set \to \Mod(\T,\Set), \]
sends each finite set $n \in \NN$ to the representable functor $\T(\tau(n),-)\maps \T \to \Set$, and in general any set $X$ to the colimit of all such representables as $n$ ranges over the poset of finite subsets of $X$.   In rough terms, $F(X)$ is the model of all $n$-ary operations from $\T$ on the set $X$.

If we momentarily abbreviate $\Mod(\T,\Set)$ as $\Mod$, we obtain an adjunction
\[   \Mod(F(n),\mu) = \Mod(\T(\tau(n),-),\mu) \cong \mu(\tau(n)) \cong \mu(\tau(1))^n = \Set(n,U(\mu))\] 
where the left isomorphism arises from the Yoneda lemma, and the right isomorphism from the product preservation of $\mu$.   

This adjunction induces a monad $T$ on $\Set$:
\begin{equation}
T(X) = \int^{n\in \NN} X^n \times \T(n,1).
\end{equation}
The integral here is a coend, essentially a coproduct quotiented by the equations of the theory and the equations induced by the cartesian structure of the category. This forms the set of all terms that can be constructed from applying the operations to the elements, subject to the equations of the theory.  The monad constructed this way is always \define{finitary}: that is, it preserves filtered colimits \cite{adamekrosicky}, or its action on sets is determined by its action on finite sets.

Conversely, for a monad $T$ on $\Set$, its Kleisli category $\Kl(T)$ is the category of all free algebras of the monad, which has all coproducts. There is a functor $k\maps \Set \to \Kl(T)$ that is the identity on objects and preserves coproducts. Thus,
\[ k^{\op}\maps \Set^{\op} \to \Kl(T)^{\op} \]
is a cartesian functor, and restricting its domain to $\NN^{\op}$ is a Lawvere theory $k_T$. To see what this is doing, note that:
$$\Kl(T)^{\op}(n,m) = \Kl(T)(m,n) = \Set(m,T(n))$$
where the latter is considered as $m$ $n$-ary operations in the Lawvere theory $k_T$.
When $T$ is finitary, the monad arising from this Lawvere theory is naturally isomorphic to $T$ itself.

This correspondence sets up an equivalence between the category $\Law$ of Lawvere theories and the category of finitary monads on $\Set$.  There is also an equivalence between the category $\Mod(\T,\Set)$ of models of a Lawvere theory $\T$ and the category of algebras of the corresponding finitary monad $T$.  Furthermore, all this generalizes with $\Set$ replaced by any ``locally finitely presentable'' category \cite{adamekrosicky}. For more details see \cite{barrwells,lawvere,milewski}.\\

One final point, provided to us by Mike Stay: while monads are often associated with functional programming languages such as Haskell, Lawvere theories correspond to \emph{interfaces} or abstract classes in object-oriented programming. In these one declares various constants, types, and abstract functions satisfying tests, and then one implements the interface by assigning these elements, sets, functions, and equations---precisely a model in $\Set$. While people think of monads as the main example of ``categories in programming'', in fact Lawvere theories are ubiquitous.

\section{Enrichment}
\label{sec:enrichment}

To incorporate the aspect of computation, we now turn to Lawvere theories that have hom-\emph{objects} rather than mere hom-\emph{sets}.  To do this we use enriched category theory \cite{kelly} and replace sets with objects of a cartesian closed category $\V$, called the ``enriching'' category or ``base''.    A $\V$-enriched category or \textbf{$\V$-category} $\C$ is:
\[\begin{array}{rl}
\text{a collection of objects} & \Obj(\C)\\
\text{a hom-object function} & \C(-,-)\maps \Obj(\C) \times \Obj(\C) \to \Obj(\V)\\
\text{composition morphisms} & \circ_{a,b,c}\maps\C(b,c) \times \C(a,b) \to \C(a,c) \quad \forall a,b,c \in \Obj(\C)\\
\text{identity-assigning morphisms} & id_a\maps 1_\V \to\C(a,a) \quad \forall a \in \Obj(\C)\\
\end{array}\]
such that composition is associative and unital.  A \textbf{$\V$-functor} $F \maps \C \to \D$ is:
\[\begin{array}{rl}
\text{a function} & F\maps \Obj(\C) \to \Obj(\D)\\
\text{a collection of morphisms} & F_{ab}\maps \C(a,b) \to \D(F(a),F(b)) \quad \forall a,b \in \C\\
\end{array}\]
such that $F$ preserves composition and identity.  A \textbf{$\V$-natural transformation} $\alpha\maps F \Rightarrow G$ is:
\[\begin{array}{rl}
\text{a family} & \alpha_a \maps 1_\V \to \D(F(a),G(a)) \quad \forall a \in \Obj(\C)\\
\end{array}\]
such that $\alpha$ is ``natural'' in $a$: an evident square commutes.   There is a 2-category \textbf{$\V\Cat$} of $\V$-categories, $\V$-functors, and $\V$-natural transformations. 

We can construct new $\V$-categories from old ones by taking products or opposites, in obvious ways.   There is also a $\V$-category denoted $\underline{\V}$ with the same objects as 
$\V$ and with hom-objects given by the internal hom:
\[   \underline{\V}(v,w) = w^v   \quad \forall v,w \in \V  .\]
The concepts of adjunction and monad generalize straightforwardly to $\V$-categories,
and when we speak of an adjunction or monad in the enriched context this generalization
is what we mean \cite{kelly}.   For example, there is an adjunction
\[    \underline{\V}(u \times v, w) \cong \underline{\V}(u, w^v ) \]
called ``currying''.

We can generalize products and coproducts to the enriched context.
Given a $\V$-category $\C$, the \textbf{$\V$-coproduct} of an $n$-tuple of objects $b_1, \dots, b_n \in \Obj(C)$ is an object $b$ equipped with a $\V$-natural isomorphism
\[           C(b,-) \cong \prod_{i = 1}^n C(b_i,-). \]
If such an object exists, we denote it by $\sum_{i=1}^n b_i$.  
This makes sense even when $n = 0$: a 0-ary $\V$-coproduct in $\C$ is called a \textbf{$\V$-initial object} and denoted as $0_\C$.   When $\V$ is cartesian
closed, any finite coproduct that exists in $\V$ is also a $\V$-coproduct in $\underline{\V}$.   In particular,
\[          u^{v+w} \cong u^v \times u^w \; \textrm{ and } \;  w^0 \cong 1_\V \]
whenever $0$ is an initial object of $\V$.  Conversely, any finite $\V$-coproduct that
exists in $\V$ is also a coproduct in the usual sense.

Similarly, a \textbf{$\V$-product} of objects $b_1, \dots , b_n \in \Obj(C)$ is an object $b$ equipped with a $\V$-natural isomorphism
\begin{equation}
\label{eq:prod}          \C(-,b) \cong \prod_{i=1}^n \C(-,b_i). 
\end{equation}
If such an object $b$ exists, we denote it by $\prod_{i=1}^n b_i$.   A 0-ary product in $\C$ is called a \textbf{$\V$-terminal object} and denoted as $1_\C$.     Whenever $\V$ is cartesian closed, the finite products in $\V$ are also $\V$-products in $\underline{\V}$.  In particular,
\[           (u \times v)^w \cong u^w \times v^w \; \textrm{ and } \; 1_\V^w \cong 1_\V \]
where our chosen terminal object $1_\V$ is also $\V$-terminal.
Conversely, any finite $\V$-product in $\V$ is also a product in the usual sense. 

A general $\V$-category $\C$ does not exactly have projections from a 
$\V$-product to its factors, since given two objects $c, c' \in \Obj(\C)$ there is not, fundamentally, a \emph{set} of morphisms from $c$ to $c'$.  Instead there is the hom-object $\C(c,c')$, which is an object of $\V$.    However, any object $v$ of $\V$ has a set of \define{elements}, namely morphisms $f \maps 1_\V \to v$.   Elements of $\C(c,c')$ act like morphisms from $c$ to $c'$.  

In particular, any $\V$-product $b = \prod_{i=1}^n b_i$ gives rise to elements
\[    p_i \maps 1_\V \to \C(b,b_i)  \]
which serve as substitutes for the projections in a usual product.   These elements
are defined as composites
\[    1_\V \stackrel{id_b}{\longrightarrow} \C(b,b) \stackrel{\sim}{\longrightarrow} \prod_{i=1}^n \C(b,b_i) \to \C(b,b_i)   \]
where the isomorphism comes from Eq.\ \eqref{eq:prod} and the last arrow is a projection in $\V$.

Even better, we can bundle up all these elements $p_i$ into a single element
\[    p \maps 1_\V \to \prod_{i=1}^n \C(b,b_i)  \]
which serves as a substitute for the universal cone in a usual product.  Starting from $p$ we can recover the $\V$-natural isomorphism in Eq.\ \eqref{eq:prod} as follows:
\begin{equation}
\label{eq:projection}    \C(-,b) \stackrel{\sim}{\longrightarrow} 1_\V \times \C(-,b)  
 \xrightarrow{p \times 1} \prod_{i=1}^n \C(b,b_i) \times \C(-,b) \longrightarrow \prod_{i=1}^n \C(-,b_i)
 \end{equation}
 where the last arrow is given by composition.   Thus, we say a \define{universal cone} exhibiting $b$ as the $\V$-product of objects $b_1, \dots, b_n$ is an element $p \maps 1_\V \to \prod_{i=1}^n \C(b,b_i)$ such that the $\V$-natural transformation $\C(-,b) \to \prod_{i=1}^n \C(-,b_i)$ given by Eq.\ \eqref{eq:projection} is an isomorphism.

The advantage of this reformulation is that we can say a $\V$-functor $F \maps \C \to \D$ \textbf{preserves finite $\V$-products} if for every universal cone $p \maps 1_\V \to \prod_{i=1}^n \C(b,b_i)$ exhibiting $b$ as the $\V$-product of the objects $b_i$, the composite
\[       1_\V \stackrel{p}{\longrightarrow} \prod_{i=1}^n \C(b,b_i) \xrightarrow{\prod_i F} \D(F(b),F(b_i)) \]
is universal cone exhibiting $F(b)$ as the $\V$-product of the objects $F(b_i)$.

A bit more subtly, generalizing the exponentials in $\V$, a $\V$-category $\C$ can have ``powers''.    Given $v \in \Obj(\V)$, we say an object $c^v \in \Obj(\C)$ is a \textbf{$v$-power} of $c \in \Obj(\C)$ if it is equipped with a $\V$-natural isomorphism
\begin{equation}
\label{eq:power}
 \C(-,c^v) \cong \C(-,c)^v.
\end{equation}
In the special case $\V = \Set$  this forces $c^v$ to be the $v$-fold product of copies of $c$.  
As with $\V$-products, it is useful to repackage the isomorphism of Eq.\ \eqref{eq:power} so we can say what it means for a $\V$-functor to preserve $v$-powers.   First, note that this isomorphism gives rise to an element 
\[     q \maps 1_\V \to  \C(c^v,c)^v , \]
namely the composite
\[ 1_\V \stackrel{id_{c^v}}{\longrightarrow}  \C(c^v,c^v) \stackrel{\sim}{\longrightarrow} \C(c^v,c)^v .\]
Conversely, any element $q \maps 1_\V \to  \C(c^v,c)^v$ determines a $\V$-natural transformation $e\maps C(-,c^v) \to C(-,c)^v$, and we say $e$ is a \textbf{universal cone} if this $\V$-natural transformation is an isomorphism.  Next, suppose $\C$ and $\D$ are $\V$-categories with $v$-powers.  We say a $\V$-functor $F\maps \C\to \D$ \textbf{preserves $v$-powers} if it maps universal cones to universal cones.

There are just a few more technicalities. A category is \textbf{locally finitely presentable} if it is the category of models for a finite limits theory, and an object is \define{finite} if its representable functor is \define{finitary}: that is, it preserves filtered colimits \cite{adamekrosicky}.   A $\V$-category $\C$ is \textbf{locally finitely presentable} if its underlying category $\C_0$ is locally finitely presentable, $\C$ has finite powers, and $(-)^x\maps \C_0 \to \C_0$ is finitary for all finitely presentable $x$.  The details are not crucial here: all categories to be considered are locally finitely presentable. We will use  $\V_f$ to denote the full subcategory of $\V$ of finite objects: in $\sSet$, these are simplicial sets with finitely many $n$-simplices for each $n$.

\section{Enriched Lawvere Theories}
\label{sec:enriched_lawvere}

Power introduced the notion of enriched Lawvere theory about twenty years ago, ``in seeking a general account of what have been called notions of computation'' \cite{power}. The original definition is as follows: for a symmetric monoidal closed category $(\V,\otimes,1)$, a ``$\V$-enriched Lawvere theory'' is a $\V$-category $\T$ that has powers by objects in $\V_f$, equipped with an identity-on-objects $\V$-functor 
\[  \tau\maps \underline{\V}_f^\op \to \T \]
that preserves these powers.  A ``model'' of a $\V$-theory is a $\V$-functor $\mu\maps\T \to \V$ that preserves powers by finite objects of $\V$.  There is a category $\Mod(\T,\V)$ whose objects are models and whose morphisms are $\V$-natural transformations. The monadic adjunction and equivalence of Section \ref{sec:lawvere} generalize to the enriched setting.


In this paper, however, we only consider \textit{natural number} arities, while still retaining enrichment. To do this we use the work of Lucyshyn-Wright \cite{lucyshyn-wright}, who along with Power \cite{np} has generalized Power's original ideas to allow a more flexible choice of arities.    We also limit ourselves to the case where the tensor product of $\V$ is cartesian.  This has a significant simplifying effect, yet it suffices for many cases of interest in computer science.

Thus, in all that follows, we let $(\V,\times,1_\V)$ be a cartesian closed category equipped with chosen finite coproducts of the terminal object $1_\V$, say 
\[   n_\V = \sum_{i \in n} 1_\V . \]  
Define $\NN_\V$ to be the full subcategory of $\V$ containing just these objects $n_\V$.  
There is also a $\V$-category $\underline{\NN}_\V$ whose objects are those of $\NN_\V$ and whose hom-objects are given as in $\V$.   
We define the $\V$-category of \define{arities} for $\V$ to be 
\[             \A_\V := \underline{\NN}_\V^\op   .\]
We shall soon see that $\A_\V$ has finite $\V$-products. 

\begin{definition}
\label{defn:V-theory}
We define a \textbf{$\V$-theory} $(\T,\tau)$ to be a $\V$-category $\T$ equipped with a $\V$-functor 
\[ \tau  \maps \A_\V \to \T \]
that is the identity on objects and preserves finite $\V$-products.
\end{definition}

\begin{definition}
A \textbf{model} of $\T$ in a $\V$-category $\C$ is a $\V$-functor 
\[  \mu \maps \T \to \C \]
that preserves finite $\V$-products.
\end{definition}

Just as all the objects of a Lawvere theory are finite products of a single object, we shall see that all the objects of $\T$ are finite $\V$-products of the object 
\[   t = \tau(1_\V) .\]   

\begin{definition}
\label{defn:VLaw}
We define $\V\Law$, the \define{category of} $\V$\define{-theories}, to be the category for which an object is a $\V$-theory and a morphism from $(\T,\tau)$ to $(\T',\tau')$ is a $\V$-functor $f \maps \T \to \T'$ that preserves finite $\V$-products and has $f\tau = \tau'$.   
\end{definition}

\begin{definition}
\label{defn:VMod}
For every $\V$-theory $(T,\tau)$ and every $\V$-category $\C$ with finite $\V$-products, we define $\Mod(\T,\C)$, the \define{category of models} of $(\T,\tau)$ in $\C$, to be the category for which an object is a $\V$-functor $\mu \maps \T\to \C$ that preserves finite $\V$-products and a morphism is a $\V$-natural transformation.
\end{definition}

The basic monadicity results for Lawvere theories generalize to $\V$-theories when $\V$ is  complete and cocomplete, as in the main examples we consider: $\V = \sSet, \Cat, \Pos,$ and $\Set$.   Under this extra assumption $\V\Law$ and  $\Mod(\T,\C)$ can be promoted to 
$\V$-categories, which we call $\underline{\V\Law}$ and $\underline{\Mod}(\T,\C)$.  Furthermore, there is a $\V$-functor 
\[   U \maps \underline{\Mod}(\T,\V) \to \underline{\V} \]
sending any model  $\mu \maps \T \to \V$ to its underlying object $\mu(t) \in \V$.   
Recall that monads and adjunctions make sense in $\V\Cat$, just as they do in $\Cat$.  
The $\V$-functor $U$ has a left adjoint
\[   F \maps \underline{\V} \to \underline{\Mod}(\T,\V) ,\]
and $\underline{\Mod}(\T,\V)$ is equivalent to the $\V$-category of algebras of the resulting monad $T = UF$.  More precisely:

\begin{theorem}
\label{thm:monadicity}
Suppose $\V$ is cartesian closed, complete and cocomplete, and has chosen finite coproducts of the terminal object.  Let $(\T,\tau)$ be a $\V$-theory.  Then there is a monadic adjunction
\[\begin{tikzcd}[column sep=small]
\phantom{AA} \underline{\V} \arrow[bend left,below]{rr}{F}
& {\phantom{Aa} \ld} &
\arrow[bend left,above]{ll}{\;\, U} \underline{\Mod}(\T,\V).
\end{tikzcd}\]
\end{theorem}

\begin{proof}
This follows from Lucyshyn-Wright's general theory \cite{lucyshyn-wright}, so our task is simply to explain how.   He allows $\V$ to be a symmetric monoidal category, and uses a more general concept of algebraic theory with a system of arities given by any fully faithful symmetric monoidal $\V$-functor $j \maps \J \to \underline{\V}$.   For us $\J = \underline{\NN}_\V$ and $j \maps \underline{\NN}_\V \to \underline{\V}$ is the obvious inclusion; this is his Example 3.7.

Lucyshyn-Wright defines a \textbf{$\J$-theory} to be a $\V$-functor $\tau \maps \J^\op \to \T$ that is the identity on objects and preserves powers by objects in $\J$ (or more precisely, their images under $j$).  For us $\J^\op = \A_\V$.   So, to apply his theory, we need to show that a $\V$-functor $\tau \maps \A_\V \to \T$ preserves powers by objects in $\NN_\V$ if and only if it preserves finite $\V$-products.  This is Lemma \ref{lem:powers_4} below.

He defines a model (or ``algebra'') of a $\J$-theory to be a $\V$-functor $\tau \maps \T \to \underline{\V}$ that preserves powers by objects in $\J$.   He defines a morphism of models to be a $\V$-natural transformation between such $\V$-functors.  So, to apply his theory, we also need to show that when $\J = \underline{\NN}_\V$, a $\V$-functor $\mu \maps \T \to \underline{\V}$ preserves powers by objects of $\J$ if and only if it preserves finite $\V$-products.   This is Lemma \ref{lem:powers_5} below.

A technical concept fundamental to Lucyshyn-Wright's theory is that of an \define{eleutheric} system of arities $j \maps \J \to \underline{\V}$.  This is one where the left Kan extension of any $\V$-functor $f \maps \J \to \underline{\V}$ along $j$ exists and is preserved by each $\V$-functor $\underline{\V}(x,-) \maps \underline{\V} \to \underline{\V}$.  In Example 7.5.5 he shows that $j \maps \underline{\NN}_\V \to \underline{\V}$ is eleutheric when $\V$ is countably cocomplete.  In Thm.\ 8.9 shows that when $j \maps \J \to \underline{\V}$ is eleutheric, and has equalizers, we may form the $\V$-category $\underline{\Mod}(\T,\V)$, and that the forgetful $\V$-functor
\[   U \maps \underline{\Mod}(\T,\V) \to \underline{\V} \]
is monadic.  This is the result we need.   So, our theorem actually holds whenever $\V$ is cartesian closed, with equalizers and countable colimits, and has chosen finite coproducts of the initial object.  \end{proof}

Before turning to examples, a word about Lucyshyn-Wright's construction of the left adjoint $F$ and the monad $T$ is in order.  These rely on the ``free model'' on an object $n_\V \in \V$.   This is the enriched generalization of the free model described in Section \ref{sec:lawvere}: it is the composite of $\tau^\op\maps \A_\V^\op \to \T^\op$ with the enriched Yoneda embedding $y\maps \T^\op \to [\T,\V]$:
\[
\begin{array}{rllll}
\A_\V^\op & \xrightarrow{\tau^\op} & \T^\op & \xrightarrow{y} & \left[\T,\V\right]\\
\\
n_\V & \mapsto & t^{n_\V} & \mapsto & \T(t^{n_\V},-)
\end{array}
\]
Since an object of $\V$ does not necessarily have a ``poset of finite subobjects'' over which to take a filtered colimit (as in $\Set$), the extension of this ``free model'' functor $y \tau^\op$ to all of $\V$ is specified by a somewhat higher-powered generalization: it is the left Kan extension of $y\tau^{\op}$ along $j$.
\[\begin{tikzcd}
\NN_\V \arrow[rr,"y\tau^\op"{name=y}] \arrow[swap,rd,"j"] & & \left[\T,\V\right]\\
& \V \arrow[swap,dotted, ru,"F:=\mathrm{Lan}_jy\tau^{\op}"] \arrow[Rightarrow,from=y,"\eta", shorten >=0.2cm,shorten <=.2cm]
\end{tikzcd}\]
This is the universal ``best solution'' to the problem of making the triangle commute up to a $\V$-natural transformation.  That is, for any functor $G \maps \V \to [\T,\V]$ and $\V$-natural transformation $\theta \maps y\tau^{\op} \Rightarrow Gj$, the latter factors uniquely through $\eta$.
From the adjunction between $\V$ and the category of models $\Mod(\T,\V)$ we obtain a $\V$-enriched monad
\[       T = U F \maps \V \to \V, \]
and this has a more concrete formula as an enriched coend:
\[
T(V) = \int^{n_\V\in \NN_\V} V^{n_\V} \times \T(t^{n_\V},t).
\]

We next give two examples of a rather abstract nature, where we show how $\Cat$-enriched Lawvere theories can describe categories with extra structure.   In Section \ref{sec:applications} we study examples more directly connected to operational semantics.

\begin{example}
\label{ex:1}
When $\V = \Cat$, a $\V$-category is a 2-category, so a $\V$-theory deserves to be called a \define{2-theory}.  For example, let $\T = \Th(\mathrm{PsMon})$ be the 2-theory of pseudomonoids.   A pseudomonoid \cite{pseudo} is a weakened version of a monoid: rather than associativity and unitality \textit{equations}, it has 2-isomorphisms called the associator and unitors, which we can treat as \textit{rewrite rules}.  To equate various possible rewrite sequences, these 2-isomorphisms must obey equations called ``coherence laws''. 
Power \cite{powsketch} has introduced ``enriched sketches'' as a way of presenting enriched Lawvere theories.   Informally, here is a presentation of the 2-theory for pseudomonoids:

	\[ \Th(\mathrm{PsMon}) \]
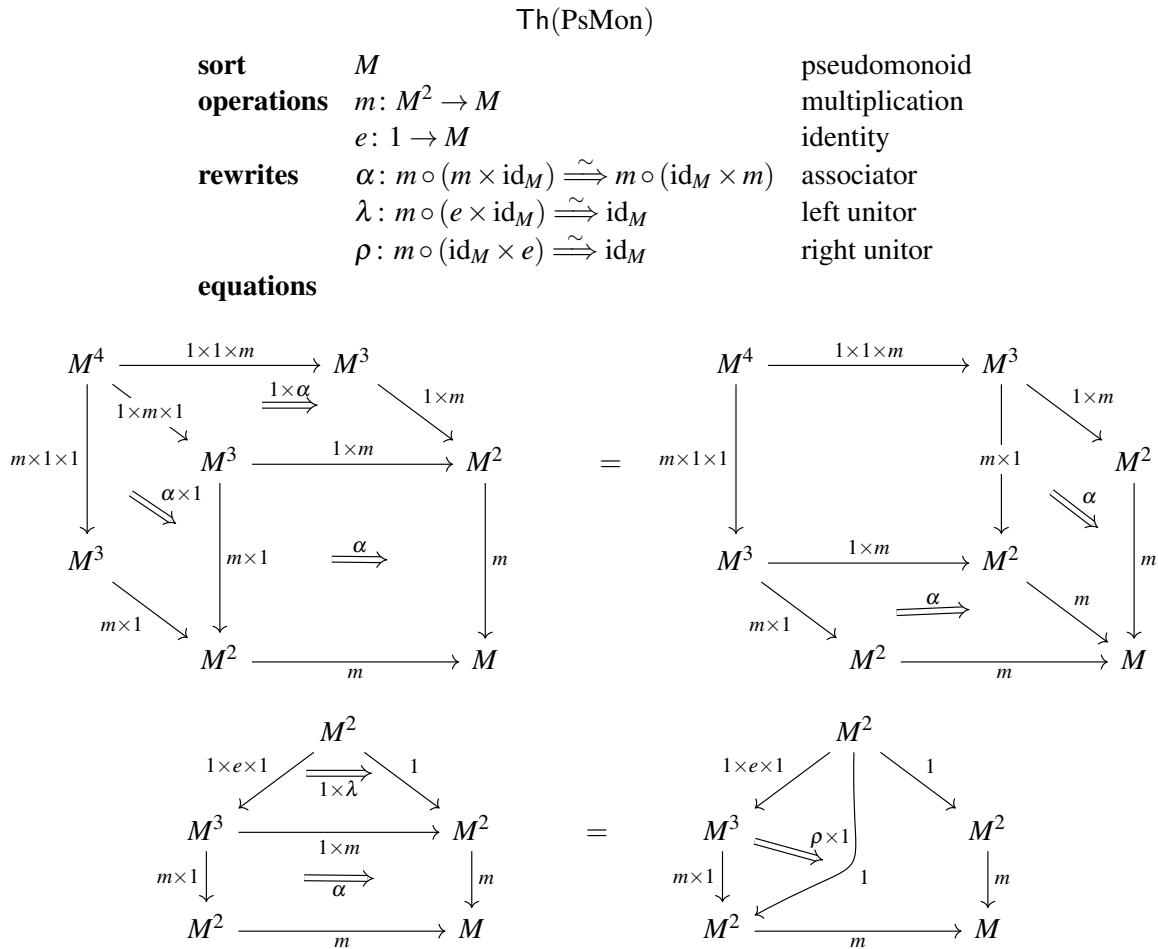
	\[\begin{array}{lll}
	\textbf{sort} & M & \text{pseudomonoid}\\
	 \textbf{operations}
	& m\maps M^2 \to M & \text{multiplication}\\
	 & e\maps1 \to M & \text{identity}\\
	\textbf{rewrites} & \alpha \colon m \circ (m \times \id_M) \stackrel{\sim}{\Longrightarrow} m \circ (\id_M \times m) & \text{associator}\\
	& \lambda\maps  m \circ (e \times \id_M) \stackrel{\sim}{\Longrightarrow} \id_M & \text{left unitor}\\
	& \rho\maps m \circ (\id_M \times e) \stackrel{\sim}{\Longrightarrow} \id_M & \text{right unitor}\\
	\textbf{equations}
          \end{array}\]

        \[\begin{tikzcd}
          M^4 \ar[rr,"1\times 1\times m"] \ar[dr,"1\times m\times 1" description, ""name=a1] \ar[dd, "m\times 1\times 1",swap,""name=b1] && M^3 \ar[dr, "1\times m",""name=a2] \arrow[Rightarrow,"1\times \alpha",from=a1,to=a2,shorten >=1.3cm,shorten <=1.3cm] &&& M^4 \ar[rr, "1\times 1\times m"] \ar[dd,"m\times 1\times 1",swap,""name=d1] && M^3 \ar[dd,"m\times 1" description,""name=d2] \ar[dr,"1\times m"] 
          & \\
          & M^3 \ar[rr,"1\times m"] \ar[dd,"m\times 1",""name=b2] \arrow[Rightarrow,"\alpha\times 1",from=b1,to=b2,shorten >=0.7cm,shorten <=0.7cm] && M^2 \ar[dd,"m",""name=c1] & = & &&& M^2 \ar[dd,"m",""name=f1] \arrow[Rightarrow,"\alpha",from=d2,to=f1,shorten >=0.6cm,shorten <=0.6cm] \arrow[Rightarrow,"\alpha",from=b2,to=c1,shorten >=1.3cm,shorten <=1.3cm] \\
          M^3 \ar[dr,"m\times 1",swap] &&& && M^3 \ar[rr,"1\times m"] \ar[dr,"m\times 1",swap,,""name=e1] && M^2 \ar[dr,"m",,""name=e2] \arrow[Rightarrow,"\alpha",from=e1,to=e2,shorten >=1.3cm,shorten <=1.3cm] & \\
          & M^2 \ar[rr,"m",swap] && M && & M^2 \ar[rr,"m",swap] && M
        \end{tikzcd}\]

      \[\begin{tikzcd}
          & M^2 \ar[dl,"1\times e\times 1",""name=a1,swap] \ar[dr,"1",""name=a2] \arrow[Rightarrow,from=a1,to=a2,"1\times \lambda",swap,shorten >=0.4cm,shorten <=0.4cm] & && & M^2 \ar[dl,"1\times e\times 1",swap] \ar[dr,"1"] \ar[ddl,bend left =30,looseness=2, "1",""name=c1] & \\
          M^3 \ar[d,"m\times 1",swap,""name=b1] \ar[rr,"1\times m",swap] && M^2 \ar[d,"m",""name=b2] \arrow[Rightarrow,from=b1,to=b2,"\alpha",shorten >=1.3cm,shorten <=1.3cm,swap] & = & M^3 \ar[d,"m\times 1",swap] \arrow[Rightarrow,to=c1,"\rho\times 1",shorten >=0.4cm,shorten <=0.0cm] && M^2 \ar[d,"m"] \\
          M^2 \ar[rr,"m",swap] && M && M^2 \ar[rr,"m",swap] && M
        \end{tikzcd}\]
\noindent We write the equations as commutative diagrams merely for convenience; they could also be written as equations in a more traditional style.   The top diagram expresses the pentagon identity for the associator, while the bottom one expresses the usual coherence law involving the left and right unitors.
   
Models of $\T = \Th(\mathrm{PsMon})$ in $\Cat$ are monoidal categories: let us explore this example in more detail.  A model of $\T$ is a finite-product-preserving 2-functor $\mu\colon \T\to \Cat$, which sends 
\[\begin{array}{rccl}
	t & \mapsto& \C & \\
	m & \mapsto & \otimes \maps &  \C^2 \to \C \\
	e & \mapsto & I \maps &  1\to \C \\
	\alpha & \mapsto & a \maps & \otimes \circ (\otimes \times 1_\C)  \Rightarrow  \otimes \circ (1_\C \times \otimes)\\
	\lambda & \mapsto & \ell \maps &  I\circ 1_\C  \Rightarrow  1_\C\\
	\rho & \mapsto & r \maps & 1_\C \circ I  \Rightarrow  1_\C
\end{array}\]
such that the coherence laws of the rewrites are preserved.  Thus, a model is a category equipped with a tensor product $\otimes$ and unit object $I$ such that these operations are associative and unital up to natural isomorphism; so these models are precisely monoidal categories.

Given two models $\mu,\nu\maps \T\to \Cat$, a morphism of models is a 2-natural transformation $\varphi\maps \mu \Rightarrow \nu$; this amounts to a strict monoidal functor $\varphi\maps (\C,\otimes_C,I_C)\to(\D,\otimes_D,I_D)$.  The strictness arises because morphisms between models are 2-natural transformations rather than pseudonatural transformations. There is a substantial amount of theory on pseudomonads and pseudoalgebras \cite{bkp,dubuc}, but to the authors' knowledge the theory-monad correspondence has not yet been extended to include the case of weak naturality. 

Finally, because $\Cat$ is complete and cocomplete, the category of models $\Mod(\T,\Cat)$ can be promoted to a 2-category $\underline{\Mod}(\T,\Cat)$.  This is the 2-category of monoidal categories, strict monoidal functors, and monoidal natural transformations.

We can accomplish the same thing on the monad side: a $\Cat$-enriched monad is called a \define{2-monad}, and $\T$ gives rise to the ``free monoidal category'' 2-monad $T$ on $\Cat$ \cite{bkp}. To apply this 2-monad to $C \in \Cat$ we first form the free model on $\C$ by taking a left Kan extension as above, and then evaluate this model at the generating object.  In the same way that the (underlying set of the) free monoid on a set $X$ consists of all finite strings of elements of $X$, $T(\C)$ is the monoidal category consisting of all finite tensor products of objects of $\C$ and all morphisms built from those of $\C$ by composition and tensoring together with associators and unitors obeying the necessary coherence laws.   Morphisms of these algebras are strict monoidal functors, while 2-morphisms are natural transformation. We thus have a 2-equivalence between $\underline{\Mod}(\T,\Cat)$ and the 2-category of algebras of $T$.





In this way, 2-theories generalize equipping \textit{set}-like objects  with operations obeying equations to equipping \textit{category}-like objects with operations obeying equations up to transformations that obey equations of their own. In particular, this gives us a way to present graphical calculi such as string diagrams -- the language of monoidal categories.
\end{example}

\begin{example}
\label{ex:2}
 Enrichment generalizes operations in more ways than by weakening equations to coherent isomorphisms.  We can also use 2-theories to describe other structures that make sense inside 2-categories, such as adjunctions. 
  
 For example, we may define a cartesian category $\X$ to be one equipped with right adjoints to the diagonal $\Delta_\X\maps \X \to \X \times \X$ and the unique functor $!_\X \maps \X \to 1_\Cat$.    These right adjoints are a functor $m \maps \X^2 \to \X$ describing binary products in $\X$ and a functor $e \maps 1 \to \X$ picking out the terminal object in $\X$.   We can capture the fact that they are right adjoints by providing them with units and counits and imposing the triangle equations.   There is thus a 2-theory $\Th(\mathsf{Cart})$ whose models in $\Cat$ are categories with chosen finite products.  More generally a model of this 2-theory in any 2-category $\C$ with finite products is called a \define{cartesian object} in $\C$.
 
  \[ \Th(\mathsf{Cart}) \]
  \[\begin{array}{lllllll}
      \textbf{type} &
        \X & \text{cartesian object}\\ \\
      \textbf{operations}  &
      m \maps \X^2 \to \X & \text{product} \\ 
        &  e \maps 1 \to \X & \text{terminal element} \\ \\
  \textbf{rewrites}  &
\bigtriangleup\maps \mathrm{id}_\X \Longrightarrow m \circ\, \Delta_\X & \text{unit of adjunction between $m$ and $\Delta_\X$} 
\\
&
\pi\maps \Delta_\X \circ m \Longrightarrow \mathrm{id}_{\X^2} & \text{counit of adjunction between $m$ and $\Delta_\X$} 
\\
& \top\maps \mathrm{id}_\X \Longrightarrow e \,\circ\, !_\X & \text{unit of adjunction between $e$ and $!_\X$} \\
& \epsilon \maps !_\X \, \circ \, e \Longrightarrow \mathrm{id}_{1} & \text{counit of adjunction between $e$ and $!_\X$} \\\\
 \textbf{equations}  
 \end{array}\]
\[\begin{tikzcd}
\Delta_\X \ar[Rightarrow,d,"\Delta_\X \circ \bigtriangleup",swap,""name=b1] \ar[Rightarrow,dr,"1"]
& &   
m \ar[Rightarrow,d,"\bigtriangleup \circ m",swap,""name=b1] \ar[Rightarrow,dr,"1"]  \\
         \Delta_X \circ m \circ \Delta_\X \ar[Rightarrow,r,"\pi \circ \Delta_\X",swap] & \Delta_\X & 
         m \circ \Delta_X \circ m \ar[Rightarrow,r,"m \circ \pi",swap] & m
        \end{tikzcd}\]
\[\begin{tikzcd}
!_\X \ar[Rightarrow,d,"!_\X \circ \top",swap,""name=b1] \ar[Rightarrow,dr,"1"]
& &   
e \ar[Rightarrow,d,"\top \circ e",swap,""name=b1] \ar[Rightarrow,dr,"1"]  \\
         !_X \circ e \, \circ \, !_\X \ar[Rightarrow,r,"\epsilon \circ !_\X",swap] & !_\X & 
         e \, \circ \, !_X \circ e \ar[Rightarrow,r,"e \circ \epsilon",swap] & e
        \end{tikzcd}\]
Again we write the equations as commutative diagrams, but this time commutative
triangles of 2-morphisms in $\Th(\mathsf{Cart})$.  These are the triangle equations that force $m$ to be the right adjoint of $\Delta_\X$ and $e$ to be the right adjoint of $!_\X$.  A model of $\Th(\mathsf{Cart})$ is a category with chosen binary products and a chosen terminal object; morphisms in $\Mod(\Th(\mathsf{Cart}),\Cat)$ are functors that strictly preserve this extra structure.

The subtle interplay between the cartesian structure of $\Th(\mathsf{Cart})$ and the cartesian structure of the object $\X \in \Th(\mathsf{Cart})$ is an example of the ``microcosm principle'': objects with a given structure are most generally defined in a context that has the same sort of structure.   As seen in the previous example, we can also define pseudomonoids in any 2-category with finite products, but this is excess to requirements: one can in fact define them more generally in any monoidal 2-category \cite{pseudo}.

In fact, if we let arities be finite categories, we would have $\Cat$-theories of categories with finite limits and colimits.   However, for the purposes of this paper we are using only natural number arities.  This suffices for constructing $\Th(\mathsf{Cart})$ and also $\Th(\mathsf{CoCart})$, the theory of categories with chosen binary coproducts and a chosen initial object.   Various other kinds of categories---distributive categories, rig categories, etc.---can also be expressed using $\Cat$-theories with natural number arities. This gives a systematic formalization of these categories, internalizes them to new contexts, and allows for the generation of 2-monads that describe them.
\end{example}

\section{Change of Base}
\label{sec:base_change}

We now have the tools to formulate the main idea: a choice of enrichment for Lawvere theories corresponds to a choice of \textit{computation}, and changing enrichments corresponds to a \textit{change of computation}. We propose a general framework in which one can translate between different forms of computation: small-step, big-step, full-step operational semantics, and denotational semantics.

\subsection{General results}
\label{ssec:base_change_results}

Suppose that $\V$ and $\W$ are enriching categories of the sort we are considering: 
cartesian closed categories equipped with chosen finite coproducts of the terminal object.
Suppose $F \maps \V \to W$ preserves finite products.   This induces a \textbf{change of base} functor $F_*\maps\V\Cat \to \W\Cat$ \cite{borceux} which takes any $\V$-category $\C$ and produces a $\W$-category $F_*(\C)$ with the same objects but with
\[      F_*(\C)(a,b) := F(\C(a,b)) \]
for all objects $a,b$.   Composition in $F_*(\C)$ is defined by
\[       F(\C(b,c)) \times F(\C(a,b)) \stackrel{\sim}{\longrightarrow} F(\C(b,c) \times \C(a,b)) 
\xrightarrow{F(\circ_{a,b,c})} F(C(a,b)) . \]  
The identity-assigning morphisms are given by
\[          1_\W \stackrel{\sim}{\longrightarrow} F(1_\V) \xrightarrow{F(id_a)} F(\C(a,a)) .\]

Moreover, if $f\maps \C \to \D \in \V\Cat$ is a $\V$-functor, there is a $\W$-functor $F_*(f) \maps F_*(C) \to F_*(D)$ that on objects equals $f$ and on hom-objects equals $F(f)$. If $\alpha\maps f \Rightarrow g$ is a $\V$-natural transformation and $c\in \Obj(\C)$, then we define
\[F_*(\alpha)_c\maps 1_\W \stackrel{\sim}{\longrightarrow} F(1_\V) \xrightarrow{F(\alpha_c)} F(\D(f(c),g(c))).\]
Thus, change of base actually gives a 2-functor from the 2-category of $\V$-categories, 
$\V$-functors and $\V$-natural transformations to the corresponding 2-category for $\W$.

In fact, the change of base operation gives a 2-functor
\[\begin{array}{ccc}
\Cart\Cat & \xrightarrow{(-)_*} & 2\Cat\\
(F\maps \V\to\W) & \mapsto & (F_*\maps \V\Cat\to\W\Cat)
\end{array}\]
where $\Cart\Cat$ is the 2-category of 
cartesian closed categories equipped with chosen finite coproducts of the terminal object, finite product preserving functors preserving these chosen finite coproducts, and natural transformations.  In particular, if $\V$ has not just finite coproducts of the terminal object but all coproducts of this object, there is a map of adjunctions
\[\begin{tikzcd}
	\Set \arrow[bend left,below]{rr}{-\cdot 1}
	& \ld &
	\arrow[bend left,above]{ll}{\V(1,-)} \V
	\arrow[maps to]{r}
	& \Cat \arrow[bend left,below]{rr}{(-\cdot 1)_*}
	& \ld &
	\arrow[bend left,above]{ll}{(\V(1,-))_*} \V\Cat.
\end{tikzcd}\]
Each set $X$ is mapped to the $X$-indexed coproduct of the terminal object in $\V$ and conversely each object $v$ of $\V$ is represented in $\Set$ by the hom-set from the unit to $v$. The latter induces the ``underlying category'' change of base, which forgets the enrichment. The former induces the ``free $\V$-enrichment'' change of base, whereby ordinary $\Set$-categories are converted to $\V$-categories, denoted $\C \mapsto \underline{\C}$. These form an adjunction, because 2-functors preserve adjunctions.

We now study how change of base affects theories and their models.  We start by asking when a functor $F \maps \V \to \W$ induces a change of base $F_*\maps\V\Cat \to \W\Cat$ that ``preserves enriched theories''.   That is, given a $\V$-theory 
\[      \tau \maps  \A_\V \to \T \]
we want to determine conditions for the base-changed functor 
\[    F_*(\tau) \maps  F_*(\A_\V) \to  F_*(\T) \]
to induce a $\W$-theory in a canonical way.   Recall that we require $\V$ and $\W$ to be
cartesian closed, equipped with chosen finite coproducts of their terminal objects.   We thus expect the following conditions to be sufficient: $F$ should be cartesian, and it should
preserve the chosen finite coproducts of the terminal object:
\[      F(n_\V) = n_\W  \]
for all $n$.   

Given these conditions there is a $\W$-functor, in fact an isomorphism
\[    \tilde{F} \maps \A_\W \to F_*(\A_\V)  . \]
On objects this maps $n_\W$ to $n_\V$, and on hom-objects it is simply the identity from
\[     \A_\W(m_\W, n_\W) = n_\W^{m_\W} = (n^m)_\W \]
to 
\[    F(\A_\V(m_\V, n_\V)) = F(n_\V^{m_\V}) = F((n^m)_\V) = (n^m)_\W \]
where we use Lemma \ref{lem:NN} in these computations.

Using this we obtain a composite $\W$-functor
\[   \A_\W \stackrel{\tilde{F}}{\longrightarrow} 
  F_*(\A_\V) \xrightarrow{F_*(\tau_\V)} F_*(\T). \]
This is the identity on objects and preserves finite $\V$-products because each of the factors
has these properties.   It is thus a $\W$-theory.   

\begin{theorem}
\label{thm:change-of-base}
Let $\V$, $\W$ be cartesian closed categories with chosen finite coproducts of their terminal objects, and let $F\maps \V \to \W$ be a cartesian functor that preserves these chosen coproducts.  Then $F_*$ \emph{preserves enriched theories}: that is, for every $\V$-theory $\tau_\V\maps \A_\V\to \T$, the $\W$-functor 
\[ \tau_\W := F_*(\tau_\V) \circ \tilde{F} \maps \A_\W \to F_*(\T)\]
is a $\W$-theory.   Moreover, $F_*$ \emph{preserves models}: for every model $\mu \maps\T \to \C$ of $(\T,\tau_\V) $, the $\W$-functor $F_*(\mu) \maps F_*(\T) \to F_*(\C)$ is a model of 
$(F_*(\T), \tau_\W)$.
\end{theorem}

\begin{proof}
We have shown the first part.  For the second, by Lemma \ref{lem:powers_5} it suffices to assume that $\mu$ preserves finite $\NN_V$-powers and check that $F_*(\mu)$ preserves $\NN_\W$-powers.  We leave this as an exercise to the reader.
\end{proof}

Hence, any cartesian functor that preserves chosen finite coproducts of the terminal object gives a change of base.  It thus provides for a method of translating formal languages between various ``modes of operation''.   Moreover, this reasoning generalizes to \textit{multisorted} $\V$-theories, enriched theories which have multiple sorts: given any $n\in \mathbb{N}$, the monoidal subcategory $(\NN_\V)^n$ is also an eleutheric system of arities, so Lucyshyn-Wright's monadicity theorem still applies.   

\subsection{Examples}
\label{sec:base_change_examples}

Now let us look at some examples.  The most important changes of base are the left adjoints in this diagram from Sec.\ \ref{sec:intro}:
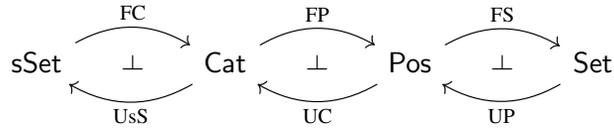
\[ \label{eq:diagram}  \begin{tikzcd}[column sep=small]
\sSet  \arrow[bend left,below]{rr}{\FC}
& \ld &
\arrow[bend left,above]{ll}{\US} 
\Cat \arrow[bend left,below]{rr}{\FP}
& \ld &
\arrow[bend left,above]{ll}{\UC} \Pos \arrow[bend left,below]{rr}{\FS}
& \ld &
\Set \arrow[bend left,above]{ll}{\UP}
\end{tikzcd}\]
The first step describes the translation from small-step to big-step operational semantics. As already mentioned, we need to use simplicial sets rather than graphs; let us now say more about why.

A first attempt might use directed multigraphs.  Such graphs have directed edges and allow multiple edges between any pair of vertices.    The category $\Gph$ of directed multigraphs is $\Set^\G$ where $\G$ is the category with two objects $v$ and $e$ and two morphisms $s,t \maps e \to v$.      The ``free category'' functor $\F \maps \Gph \to \Cat$ gives for every graph $G$ a category $\F (G)$ as follows:
\[\begin{array}{rl}
\text{objects} & \text{vertices of } G\\
\text{morphisms} & (e_1,e_2,...,e_n)\maps s(e_1)\to t(e_n) \;\; \; : \;\;\; \;\; \forall i<n \;\; t(e_i)=s(e_{i+1})\\
\text{composition} & (e_1,e_2,...,e_m) \circ (e'_1,e'_2,...,e'_n) = (e'_1,...,e'_n,e_1,...,e_m) \; : \;\; t(e_n')=s(e_1).\\
  \end{array}\]
The morphisms in $\F(G)$ are called \define{edge paths}.  Just as an edge describes a single rewrite in small-step operational sematics, an edge path describes a sequence of rewrites in big-step operational semantics.  The edge paths with $n = 0$ serve as identity morphisms.  

Unfortunately, $\F \maps \Gph \to \Cat$ does not preserve products, so it is not a valid base change.  To see this, let $G_1$ be $\{0\xrightarrow{e} 1\}$, the graph with two vertices and one edge.  The product $G_1\times G_1$ looks like this:
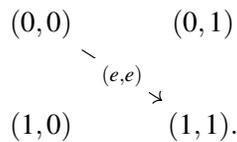
\[\begin{tikzcd}
    (0,0) \ar[dr,"(e{,}e)" description] & (0,1) \\
    (1,0) & (1,1).
  \end{tikzcd}\]
Thus the free category $\F(G_1 \times G_1)$ has just one non-identity morphism.    On the other hand $\F(G_1)\times \F(G_1)$ has five non-identity morphisms, shown here:
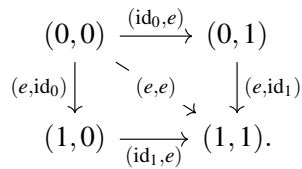
\[\begin{tikzcd}
    (0,0) \arrow[r,"(\mathrm{id}_0{,}e)"] \ar[d,"(e{,}\mathrm{id}_0)",swap] \ar[dr,"(e{,}e)" description] & (0,1) \ar[d,"(e{,}\mathrm{id}_1)"]\\
    (1,0) \ar[r,"(\mathrm{id}_1{,}e)",swap] & (1,1).
  \end{tikzcd}\]
where we write $\mathrm{id}$ for identity morphisms and $e$ for the edge path consisting of the single edge $e$.   Note that the triangles in this diagram commute.  In terms of rewriting, the category $\F(G_1 \times G_1)$ only allows the rewrite $e \maps 0 \to 1$ to occur simultaneously in both factors, while $\F(G_1)\times \F(G_1)$ allows it to occur independently in either factor, in a commuting way.

To solve this problem one, might try to use reflexive graphs.    Such graphs have directed edges and allows multiple edges between any pair of vertices; further, each vertex $v$ is equipped with a distinguished \define{identity edge} $i(v)$ from $v$ to itself. The category $\RGph$ of reflexive graphs is $\Set^\R$, where $\R$ is the category with two objects $v$ and $e$, two morphisms $s,t \maps e \to v$, and a morphism $i \maps v \to e$ obeying $si = ti = 1_v$.   There is a free category functor $\F' \maps \RGph \to \Cat$, which is like the free category functor for $\Gph$ except that we identify an edge path $(e_1, \dots, e_n)$ with the same path having $e_i$ omitted when $e_i$ is an identity edge.  Thus, the identity edges of a reflexive graph $R$ become identity morphisms in $\F'(R)$.   

The advantage of reflexive graphs is that they allow rewrites in a product to occur independently in either factor.   For example, let $R_1$ be the reflexive graph with two vertices and one non-identity edge, $\{0\xrightarrow{e} 1\}$ (where we do not draw identity edges).  The product $R_1 \times R_1$ has five non-identity edges:
\[\begin{tikzcd}
    (0,0) \arrow[r,"(i(0){,}e)"] \ar[d,"(e{,}i(0))",swap] \ar[dr,"(e{,}e)" description] & (0,1) \ar[d,"(e{,}i(1))"]\\
    (1,0) \ar[r,"(i(1){,}e)",swap] & (1,1).
  \end{tikzcd}\]
Thus, the free category $\F'(R_1 \times R_1)$ has two \emph{noncommuting} triangles.  On the other hand, $\F'(R_1) \times \F'(R_1)$ is the product of the category with a single non-identity morphism $e \maps 0 \to 1$ with itself, so it is this category:
\[\begin{tikzcd}
    (0,0) \arrow[r,"(\mathrm{id}_0{,}e)"] \ar[d,"(e{,}\mathrm{id}_0)",swap] \ar[dr,"(e{,}e)" description] & (0,1) \ar[d,"(e{,}\mathrm{id}_1)"]\\
    (1,0) \ar[r,"(\mathrm{id}_1{,}e)",swap] & (1,1)
  \end{tikzcd}\]
 with two commuting triangles.  Thus $\F' \maps \RGph \to \Cat$ again fails to preserve products, though in some sense it comes closer.  Simply put, while $\F'(R_1 \times R_1)$ allows rewrites to be done independently in either factor, these rewrites fail to commute.
 
To solve this problem we shall consider $\RGph$ as a full subcategory of the category of simplicial sets, $\sSet$.   To do this, we treat a reflexive graph as a simplicial set  with only degenerate simplices for $n > 1$.  There is a left adjoint $\FC \maps \sSet \to \Cat$, usually called \define{realization}, and this functor preserves products \cite[Prop.\ B.0.15]{joyal}.  For example, if we treat $R_1$ above as a simplicial set and take the product $R_1 \times R_1$ in $\sSet$, this contains triangles that force the triangles in
$\FC(R_1 \times R_1)$ to commute.    Thus, realization provides a useful base change to translate from small-step operational semantics to big-step operational semantics.

The other functors in our chain of left adjoints are simpler.  The ``free poset'' functor $\FP\maps \Cat \to \Pos$ maps any category $C$ to the poset whose elements are objects of $C$, with $c \le c'$ iff $C$ contains a morphism from $c$ to $c'$.  This is a valid change of base---i.e., it preserves finite products---because the product of posets is defined in the same way as the product of categories.    If we apply this change of base to a model of a $\Cat$-enriched theory, we obtain a model of a $\Pos$-enriched theory that says for any pair of terms the presence or absence of a rewrite sequence from one to the other, without distinguishing between different sequences.  We call this \define{full-step operational semantics}.

Finally, we can pass to the purely abstract realm where all computation is already complete.  For this we take the left adjoint $\FS\maps \Pos \to \Set$ to the functor $\UP \maps \Set \to \Pos$ sending any set to the discrete poset on that set.  The functor $\FS$ collapses each connected component of the poset to a point; this too preserves finite products.  If we apply this change of base to a model of a $\Pos$-enriched theory, we obtain a model of a $\Set$-enriched theory that extracts its denotational semantics by identifying all terms related by rewrites.  If the rewrites are terminating and confluent, we can choose a representative term for each equivalence class: the unique term that admits no nontrivial rewrites.

\section{The Category of All Models}
\label{sec:V-theories}

In addition to base change, there are two other natural and useful ways to go between models of enriched theories.   Suppose $\V$ is any cartesian closed category with chosen finite coproducts of the terminal object.   Let $\V\Mod(\T,\C)$ be the category of models of a $\V$-theory $\T$ in a $\V$-category $\C$ with finite $\V$-products, as in Defn.\ \ref{defn:VMod}.  A morphism of $\V$-theories $f\maps\T\to \T'$ induces a \define{change of theory} functor between the respective categories of models 
\[  f^*\maps\V\Mod(\T',\C)\to \V\Mod(\T,\C) \]
defined as pre-composition with $f$.  Similarly, a $\V$-product-preserving $\V$-functor $g\maps \C \to \C'$ induces a \define{change of context} functor 
\[  g_*\maps \V\Mod(\T,\C) \to \V\Mod(\T,\C') \]
defined as post-composition with $g$.   

These translations, as well as change of base, can all be packed up nicely using the \textit{Grothendieck construction}: given any functor $F\maps \D \to \Cat$, there is a category $\int F$ that encapsulates all of the categories in the image of $F$, defined
as follows:
\[\begin{array}{rl}
\text{objects} & (d,x) \colon d\in \D, \; x\in F(d)\\
\text{morphisms} & (f\maps d\to d',a\maps F(f)(x)\to x')\\
\text{composition} & (f,a) \circ (f',a') = (f \circ f', a \circ F(f)(a')).
\end{array}\]
Moreover there is a functor $p_F \maps \int F \to \D$ given as follows:
\[\begin{array}{rl}
\text{on objects} & p_F \maps (d,x) \mapsto d \\
\text{on morphisms} & p_F \maps (f,a) \mapsto f .\\
\end{array}\]
For more details see \cite{borceux,jacobs}. We noted in Section \ref{sec:enriched_lawvere} that $\V\Law$ and $\Mod(\T,\C)$ can be promoted to $\V$-categories when $\V$ is complete and cocomplete: this and further conditions imply that we can use the enriched Grothendieck construction \cite{beardsleywong}, but we focus on the ordinary Grothendieck
construction for simplicity.   

First, this construction lets us bring together all models of all different $\V$-theories in all different contexts into one category.   All the $\V$-theories
are objects of $\V\Law$, as in Defn. \ref{defn:VLaw}.   We can also create a category of all  ``$\V$-contexts''.

\begin{definition} 
\label{defn:VCon}
Let $\V\Con$, the \textbf{category of $\V$-contexts} be the category for which an
object is a $\V$-category with finite $\V$-products and  a morphism is a functor that
preserves finite $\V$-products.   
\end{definition}

There is a functor 
\[    \V\Mod\maps \V\Law^\op \times \V\Con \to \Cat \]
that sends any object $(\T,\C)$ to $\V\Mod(\T,\C)$ and any morphism $(f,g)$ to
$f^* g_* = g_* f^*$.   The functoriality of $\V\Mod$ summarizes the 
contravariant change-of-theory and the covariant change-of-context above.
Applying the Grothedieck construction we obtain a category $\int \V\Mod$. 
Technically an object of $\int \V\Mod$ is a triple $(T,\C,\mu)$, but more intuitively 
it is a model $\mu \maps \T \to \C$ of any $\V$-theory $\T$ in any $\V$-context $\C$.   Similarly, a morphism 
\[   (f,g,\alpha)\maps (\T,\C,\mu) \to (T',\C',\mu') \]
in $\V\Mod$ consists of:
\begin{itemize}
\item
a morphism of $\V$-theories $f \maps \T' \to \T$,
\item
a $\V$-functor $g \maps\C\to \C'$ that preserves finite $\V$-products, and
\item 
a $\V$-natural transformation $\alpha\maps g \circ \mu \circ f \Rightarrow \mu'$.
\end{itemize}
This is a natural way to map between different models of different theories in different contexts.

We can go further by creating a category that even contains all choices of enriching 
categories $\V$:

\begin{definition}
\label{defn:enrichment}
Let $\Enr$ be the category for which an object is a cartesian closed category $\V$ with chosen finite coproducts of the terminal object, and a morphism is a cartesian
 functor $F \maps \V \to \W$ preserving the chosen finite coproducts of 
the initial object.   
\end{definition}

There is a functor 
\[   \mathrm{Mod} \maps \Enr \to \Cat \]
that maps any object $\V$ to $\int \V\Mod$ and any morphism $F \maps \V \to \W$ to a functor 
\[    \mathrm{Mod}(F) \maps \textstyle{\int \V\Mod \to \int \W\Mod} \]
that has the following effect:
\begin{itemize}
\item  $\mathrm{Mod}(F)$ maps any object $(\T,\C,\mu)$ to
the object $(F_*(\T),F_*(\C),F_*(\mu))$. 
\item $\mathrm{Mod}(F)$ maps any morphism $ (f,g,\alpha)$
to the morphism $(F_*(f),F_*(g),F_*(\alpha))$. 
\end{itemize}
Thus, we can use the Grothendieck construction once more to pack up all choices of 
enrichment into one big category:

\begin{theorem}
There is a category $\int \mathrm{Mod}$ in which:
\begin{itemize}
\item An object is a choice of cartesian closed category $\V$ with chosen finite coproducts
of the terminal object, a $\V$-theory $\T$, a  $\V$-category $\C$ with finite $\V$-products, and a model $\mu \maps \T \to \C$.
\item A morphism is a cartesian functor $F \maps \V \to \W$ preserving the chosen finite
coproducts of the terminal object and a morphism $(f,g,\alpha) \maps (F_*(\T), F_*(\C), F_*(\mu)) \to (\T,\C,\mu)$ in $\W\Mod$.
\end{itemize}
\end{theorem}

This category allows us to formally treat morphisms between objects of ``different kinds'', something we often use informally, for example when speaking of a map from a set to a ring, or a group to a topological group.   There are many unexplored questions about the large, heterogeneous categories which arise from the Grothendieck construction, regarding what unusual structure may be gained, such as limits and colimits with objects of different types, or identifying ``processes'' in which the kinds of objects change in an essential way. However, for our purposes we need only recognize that enriched Lawvere theories can be assimilated into one category, providing a single place in which to study change of base, change of theory, and change of context.

\section{Applications}
\label{sec:applications}

In computer science literature, enriched algebraic theories have primarily been studied in the context of ``computational effects'' \cite{compeffects}.   Stay and Meredith have proposed that enriched Lawvere theories can be utilized for the design of programming languages \cite{ladl}. To circumvent the question of \textit{variable binding}, there is another approach which instead uses an enriched theory as a ``compiler'' which translates a language with binding to one without. This idea comes from the subject of combinatory logic. 


\subsection{The \texorpdfstring{$SKI$}{SKI}-combinator calculus}
\label{ssec:SKI}

The $\lambda$-calculus is an elegant formal language which is the foundation of functional computation, the model of intuitionistic logic, and the internal logic of cartesian closed categories: this is the Curry--Howard--Lambek correspondence \cite{baezstay}.

Terms are constructed recursively by \textit{variables}, \textit{application}, and \textit{abstraction}, and the basic rewrite is \textit{beta reduction}, which substitutes the applied term for the bound variable: 
\[ \begin{array}{l}
     M,N := x \;\; | \;\; (M\; N) \;\; | \;\; \lambda x.M\\\\
     (\lambda x.M\; N) \Rightarrow M[N/x].
     \end{array}\]
Despite the apparent simplicity, there are complications regarding substitution. Consider the term $M = \lambda x.(\lambda y.(xy))$: if this is applied to the variable $y$, then $(M\; y) \Rightarrow \lambda y.(y\; y)$ --- but this is not intended, because the $y$ in $M$ is just a placeholder, it is ``bound'' by whatever will be plugged in, while the $y$ being substituted is ``free'', meaning it can refer to some other value or function in the program. Hence whenever a free variable is to be substituted for a bound variable, we need to rename the bound variable to prevent ``variable capture'' (e.g. $(M y) \Rightarrow \lambda z.(y\; z)$).

This problem was noticed early in the history of mathematical foundations, even before the $\lambda$-calculus, and so Moses Sch\"onfinkel invented \textbf{combinatory logic} \cite{combs}, a basic form of logic without the red tape of variable binding, hence without functions in the usual sense. The $SKI$-calculus is the ``variable-free'' representation of the $\lambda$-calculus; $\lambda$-terms are translated via ``abstraction elimination'' into strings of combinators and applications. This is a technique for programming languages to minimize the subtleties of variables. 

The insight of Stay and Meredith \cite{roswelt} is that even though enriched Lawvere theories have no variables, they can be used to study some programming languages through abstraction elimination.  When representing such a language as a $\sSet$-theory, vertices---i.e., 0-simplices---in the simplicial set $\hom(1,t)$ serve as closed terms. More generally, vertices in $\hom(t^n,t)$ serve as terms with $n$ free variables.  Rewrite rules going between such terms are edges---i.e., 1-simplices---in $\hom(t^n,t)$. 

To illustrate this, here is the theory of the $SKI$-calculus:

\[  \Th(\mathsf{SKI})\]
\[\begin{array}{lrl}
\textbf{type} & t &\\
\textbf{term constructors} & S\maps & 1 \to t\\
& K\maps & 1 \to t\\
& I\maps & 1 \to t\\
& (-\; -)\maps &  t^2 \to t\\
\textbf{structural congruence} & \text{n/a} &\\
\textbf{rewrites} & \sigma\maps & (((S\; -)\; =)\; \equiv) \Rightarrow ((-\; \equiv)\; (=\; \equiv))\\
& \kappa\maps & ((K\; -)\; =) \Rightarrow -\\
& \iota\maps & (I\; -) \Rightarrow -\\
\end{array}\]
These rewrites are implicitly universally quantified; i.e., they apply to arbitrary subterms $-, =, \equiv$ without any variable binding involved, by using the cartesian structure of the category. They are edges with vertices as follows:
\[\begin{tikzcd}
(((S\; -)\; =)\; \equiv)\maps \arrow[Rightarrow,d,swap,"\sigma"] & t^3 \arrow[d,equal] \arrow[r,"l^{-1} \times t^3"] & 1\times t^3 \arrow[r,"S \times t^3"] & t^4 \arrow[Rightarrow,d,shorten >=0.2cm,shorten <=0.2cm] \arrow[r,"(-\;-)\times t^2"] & t^3 \arrow[r,"(-\;-) \times t"] & t^2 \arrow[r,"(-\;-)"] & t\arrow[d,equal]\\
((-\; \equiv)\; (=\; \equiv))\maps & t^3 \arrow[r,"t^2 \times \Delta"] & t^4 \arrow[r,"t \times s \times t"]& t^4 \arrow[rr,"(-\;-) \times (-\;-)"]&& t^2 \arrow[r,"(-\;-)"]& t\\
((K\; -)\; =)\maps \arrow[Rightarrow,d,swap,"\kappa"] & t^2 \arrow[d,equal] \arrow[r,"l^{-1} \times t^2"] & 1\times t^2 \arrow[r,"K \times t^2"] & t^3 \arrow[Rightarrow,d,shorten >=0.2cm,shorten <=0.2cm] \arrow[r,"(-\;-)\times t"] & t^2 \arrow[r,"(-\;-)"] & t \arrow[d,equal]\\
-\maps & t^2 \arrow[rr,"t \times !"] && t \times 1 \arrow[rr,"r"] && t\\
(I\; -)\maps \arrow[Rightarrow,d,swap,"\iota"] & t \arrow[d,equal] \arrow[r,"l^{-1}"] & 1\times t \arrow[r,"I \times t",""name=1] & t^2 \arrow[r,"(-\;-)"] & t \arrow[d,equal]\\
-\maps & t \arrow[rrr,"t",""name=2] \arrow[Rightarrow,from=1,to=2,start anchor={[xshift=-0.8ex]},shorten >=0.3cm,shorten <=0.3cm] &&& t
\end{tikzcd}\]
Here $l,r$ denote the unitors and $s$ the symmetry of the product.

These abstract rules are evaluated on concrete terms by ``plugging in'' via precomposition.   For example:
\[\begin{tikzcd}
	((KS)I)\maps \arrow[Rightarrow,d,swap,"\kappa \circ (S\times I)"] & 1 \arrow[d,equal] \arrow[rr,"S\times I"] && t^2 \arrow[Rightarrow,d,shorten >=0.2cm,shorten <=0.2cm] \arrow[rr,"((K\; -)\; =)"] && t \arrow[d,equal]\\
	S\maps & 1 \arrow[rr,"S\times I"] && t^2 \arrow[rr,"-"] && t
\end{tikzcd}\]

A model of this theory is a $\sSet$-functor $\mu\maps \Th(\mathsf{SKI}) \to \sSet$ that preserves finite $\sSet$-products. This gives a simplicial set $\mu(t)$.
The images of the nullary operations $S,K,I$ under $\mu$ are distinguished vertices of $\mu(t)$, because $\mu$ preserves the terminal object, which ``points out'' vertices. The image of the binary operation $(-\; -)$ gives for every pair of vertices $(u,v) \in \mu(t)^2$ a vertex $(u\; v)$ in $\mu(t)$ which stands for their application. In this way all possible terms built from $S$, $K$, $I$ and application give vertices in $\mu(t)$.  Similarly, rewrites going between these terms give edges in $\mu(t)$.  Thus, $\mu$ gives a map of simplicial sets
\[  \mu_{1,t} \maps \Th(\mathsf{SKI})(1,t)\to \sSet(1,\mu(t)) \] 
that maps the ``syntactic'' graph of all closed terms and rewrites to the ``semantic'' graph: each rewrite between terms in the theory is sent to a rewrite between the images of these terms in the model.   

The fact that $\mu((-\;-)): \mu(t)^2\to \mu(t)$ is not just a function but a map of simplicial sets means that pairs of edges $(a\to b, c\to d)$ in $\Th(\SKI)(1,t)$ are sent to edges $(a\; b) \to (c\; d)$ in $\sSet(1,\mu(t))$.  This gives the full complexity of the theory: given a large term (program), there are many different ways it can be computed---and some take fewer steps than others:
\[\begin{tikzcd}
	((K\; S)\; (((S\; K)\; I)\; (I\; K))) \arrow[r,"((K\; S)\; \sigma)"] \arrow[dddd,swap,"\kappa"] & ((K\; S)\; ((K\;(I\; K))\; (I\; (I\; K)))) \arrow[d,"(((K\; S)\; \iota)\; (I\; (I\; K)))"]\\
	& ((K\; S)\; ((K\; K)\; (I\; (I\; K)))) \arrow[d,"((K\; S)\; ((K\; K)\; (I\; \iota)))"]\\
	& ((K\; S)\; ((K\; K)\; (I\; K))) \arrow[d,"((K\; S)\; ((K\; K)\; \iota))"]\\
	& ((K\; S)\; ((K\; K)\; K)) \arrow[d,"((K\; S)\; \kappa)"]\\
	S & ((K\; S)\; K)\arrow[l,"\kappa"]
      \end{tikzcd}\]
    
More generally, the image $\mu(t)^n$ is a simplicial set whose vertices are $SKI$-terms with $n$ free variables and whose edges are $n$-tuples of rewrites between such terms.  This is because the enriched functor $\mu$ gives maps of simplicial sets
\[  \mu_{t^n,t} \maps \Th(\mathsf{SKI})(t^n,t)\to \sSet(\mu(t)^n,\mu(t)). \] 
As the $n$-ary operations and rewrites thereof are built up from application and the three rewrites, everything works the same way as in the case $n = 0$. 


This process is intuitive, but how do we actually define the model, as a functor, to pick out a specific graph? There are many models of $\Th(\mathsf{SKI})$, but in particular we care about the canonical \textit{free} model, which means that $\mu(t)$ is simply the graph of all closed terms and rewrites in the $SKI$-calculus. This utilizes the enriched adjunction of Thm.\ \ref{thm:monadicity}:
\[\begin{tikzcd}
\phantom{AAAA} \underline{\sSet} \arrow[bend left=20,"f_\sSet"]{rr}
& {\phantom{Aa} \ld} &
\arrow[bend left=20,"u_\sSet"]{ll} \underline{\Mod}(\Th(\mathsf{SKI}),\sSet)
\end{tikzcd}\]

Then the canonical model of closed terms and rewrites is simply the free model on the empty graph, $f_\sSet(\emptyset)$, i.e. the $\V$-functor $\T(1,-)\maps\T\to \V$. Hence for us, the syntax and semantics of the $SKI$ combinator calculus are unified in the model 
\[ \mu_{SKI}^\sSet:= \Th(\mathsf{SKI})(1,-)\maps \Th(\mathsf{SKI}) \to \sSet. \]   
Here we reap the benefits of the abstract construction: the graph $\mu_{SKI}^\sSet(t)$ represents the small-step operational semantics of the $SKI$-calculus: 
\[ (\mu(a) \to \mu(b) \in \mu_{SKI}^\sSet(t)) \iff (a \Rightarrow b \in \Th(\mathsf{SKI})(1,t)).\]


We can now consider the base changes in Sec.\ \ref{sec:base_change_examples}, to translate between several important kinds of computation for the $SKI$-calculus. Given the above description of $\Th(\SKI)$ as enriched in $\sSet$, we can apply the ``free category'' realization functor to the hom-objects, turning these reflexive graphs into categories.

Here we enjoy the fact that this functor indeed preserves products, which is essential for considering tuples of programs running in parallel: for example if we designate $G_n := \Th(\SKI)(t^n,t)$, then the fact that $\FC(G_m\times G_n)\cong \FC(G_m)\times \FC(G_n)$ ensures that the execution of an $m$-term program and an $n$-term program simultaneously (but independently) is the same as executing one, then the other.

Thus $\FC$ translates the theory of $SKI$ from ``small-step'' to ``big-step'' operational semantics: \newline $\FC_*(\Th(\SKI))$ is the theory of the $SKI$ calculus, but now with hom-categories whose morphisms represent finite sequences of rewrite edges in the original theory.

We can continue these base-changes to ``full-step'' and denotational semantics, by applying the ``free poset'' and ``free set'' (connected components) functors to the hom-objects of this theory. This process demonstrates the idea of having a ``spectrum'' of detail with which to analyze the semantics of a programming language, or general algebraic theory.

For example, consider the following computation:
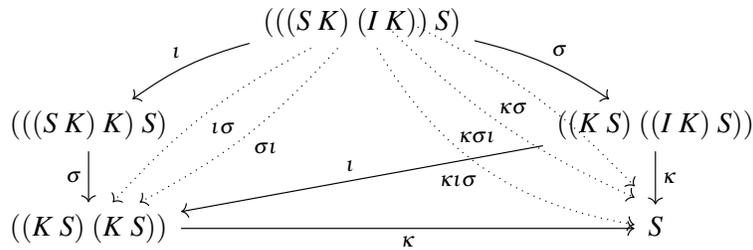
\[\begin{tikzcd}
&	(((S\; K)\; (I\; K))\; S) \arrow[rd,bend left=10,"\sigma"] \arrow[ld,bend right=10,swap,"\iota"] \arrow[ddl,dotted,bend left=10,"\sigma\iota"] \arrow[ddl,dotted,bend right=10,"\iota\sigma"] \arrow[ddr,dotted,swap,bend left=10,yshift=8pt,"\kappa\sigma"] \arrow[ddr,dotted,bend right=10,swap,yshift=8pt,"\kappa\sigma\iota"] \arrow[ddr,dotted,swap,bend right=25,"\kappa\iota\sigma"] &\\
(((S\; K)\; K)\; S) \arrow[d,swap,"\sigma"] & & ((K\; S)\; ((I\; K)\; S)) \arrow[lld,swap,"\iota"] \arrow[d,"\kappa"]\\
((K\; S)\; (K\; S)) \arrow[rr,swap,"\kappa"] & & S
\end{tikzcd}\]
The solid arrows are the one-step rewrites of the initial $\sSet$-theory; applying $\FC_*$ gives the dotted composites, and $\FP_*$ asserts that all composites between any two objects are equal. Finally, $\FS_*$ collapses the whole diagram to $S$. This is a simple demonstration of the basic stages of computation: small-step, big-step, full-step, and denotational semantics.

\subsection{Change of theory}

We can equip term calculi with \textit{reduction contexts}, which determine when rewrites are valid, thus giving the language a certain \textbf{evaluation strategy}. For example, the ``weak head normal form'' is given by only allowing rewrites on the left-hand side of the term.

We can do this for $\Th(\mathsf{SKI})$ by adding a reduction context marker as a unary operation, and a structural congruence rule which pushes the marker to the left-hand side of an application; lastly we modify the rewrite rules to be valid only when the marker is present:
$$\Th(\mathsf{SKI}+\mathsf{R})$$
\[\begin{array}{lrl}
\textbf{sort} & t &\\
\textbf{term constructors} & S,K,I \maps &1 \to t\\
& R\maps & t \to t\\
& (-\; -)\maps & t^2 \to t\\\\
    \textbf{structural congruence} & R(-\; =) & = (R-\; =) \\
& RR & = R\\\\
    \textbf{rewrites} & \sigma_r\maps & (((RS\; -)\; =)\; \equiv) \Rightarrow ((R-\; \equiv)\; (=\; \equiv))\\
& \kappa_r\maps & ((RK\; -)\; =) \Rightarrow R-\\
              & \iota_r\maps & (RI\; -) \Rightarrow R-\\
\end{array}\]

The $SKI$-calculus is thereby equipped with ``lazy evaluation'', an essential paradigm in modern programming. This represents a broad potential application of equipping theories with computational methods, such as evaluation strategies.

Moreover, these equipments can be added or removed as needed: using change-of-theory, we can utilize a ``free reduction'' $\sSet$-functor $f_R\maps\Th(\mathsf{SKI})\to \Th(\mathsf{SKI}+\mathsf{R})$:
\[\begin{array}{rrcl}
\text{objects} & t^n & \mapsto & t^n\\
\text{hom-vertices} & S,K,I & \mapsto & S,K,I\\
& (-\; -) & \mapsto & R(-\; -)\\
\text{hom-edges} & \sigma, \kappa, \iota & \mapsto & \sigma_r, \kappa_r, \iota_r\\
\end{array}\]
This essentially interprets ordinary $SKI$ as having every subterm be a reduction context. This is a $\sSet$-functor because its hom component consists of graph-homomorphisms $$f_{n,m}\maps \Th(\mathsf{SKI})(t^n,t^m) \to \Th(\mathsf{SKI}+\mathsf{R})(t^n,t^m)$$ which simply send each application to its postcomposition with $R$, and each rewrite to its ``marked'' correspondent.

So, by precomposition this induces the change of theory on categories of models: 
\[   f_R^*\maps \Mod(\Th(\mathsf{SKI} + \mathsf{R}),\C) \to \Mod(\Th(\mathsf{SKI}),\C) \]
for all semantic categories $\C$, which forgets the reduction contexts.

Similarly, there is a $\sSet$-functor $u_R\maps \Th(SKI+R)\to \Th(SKI)$ which forgets reduction contexts, by sending $\sigma_r,\kappa_r,\iota_r \mapsto\sigma, \kappa,\iota$ and $R \mapsto id_t$; this latter is the only way that the marked reductions can be mapped coherently to the unmarked. However, this means that $u_R^*$ does not give the desired change-of-theory of ``freely adjoining contexts'', because collapsing $R$ to the identity eliminates the significance of the marker.

This illustrates a key aspect of categorical universal algebra: because change-of-theory is given by precomposition and is thus contravariant, \textit{properties} (equations) and \textit{structure} (operations) can only be removed. This is a necessary limitation, at least in the present setup, but there are ways to make do. These abstract theories are not floating in isolation but are implemented in code: one can simply use a ``maximal theory'' with all pertinent structure, then selectively forget as needed.

\section{Conclusion}
\label{sec:conclusion}

We have shown how enriched Lawvere theories provide a framework for unifying the structure and behavior of formal languages. Enriching theories in category-like structures reifies operational semantics by incorporating rewrites between terms, and appropriate functors between enriching categories induce change-of-base functors between categories of enriched theories and models---this simplified condition is obtained by using only natural number arities. This idea is motivated by an example sequence of such functors, which provide a spectrum of detail in which to study the rewriting properties of a theory.

Change of base, along with change of theory and change of context, can be used to create a single category $\mathrm{Mod}$, which consists of all models of all enriched Lawvere theories in all contexts. We have demonstrated these concepts with the theory of combinatory logic, $\Th(\mathsf{SKI})$, describing a change of base from small-step operational semantics to big-step to full-step to denotational semantics.

Finally, we suggest that there are many interesting change-of-base functors, by considering an endofunctor on the category of labelled transition systems, which quotients by the bisimulation relation and is indeed a change of base.

\newpage

\appendix

\section{Natural Number Arities}
\label{sec:arities}

In this appendix we prove the lemmas required for Theorem \ref{thm:monadicity} and our study of base change in Section \ref{sec:base_change}.  Throughout we assume $\V$ is cartesian closed with chosen $n$-fold coproducts $n_\V$ of its terminal object. 

We begin with a study of $\NN_\V$, the full subcategory of $\V$ on the objects $n_\V$.
First we must resolve a potential ambiguity.  On the one hand,
for any object $b$ of $\V$ we can form the exponential $b^{n_\V}$.   On the 
other hand, we can take the product of $n$ copies of $b$, which we call $b^n$.    
Luckily these are the same, or at least naturally isomorphic:

\begin{lemma}
\label{lem:powers_1}
The functors $(-)^{n_\V} \maps \V\to \V$ and $(-)^n\maps \V\to \V$ are naturally isomorphic.  
\end{lemma}
\begin{proof}
If $a,b \in \V$, then
\[\begin{array}{rcll}
	\V(a,b^{n_\V}) & \cong & \V(a\times n_\V,b) & \text{hom-tensor adjunction}\\
	& = & \V(a\times (n \cdot 1_\V),b) & \text{definition of } n_\V\\ 
	& \cong & \V(n \cdot( a\times 1_\V), b) & \text{products distribute over coproducts}\\  
	& \cong & \V(n \cdot a,b) & \text{unitality}\\  
	& \cong & \V(a,b)^n & \text{definition of coproduct}\\
	& \cong & \V(a,b^n) & \text{definition of product}.\\
	\end{array}
\]
Each of these isomorphisms is natural in $a$ and $b$, so by the Yoneda lemma $(-)^{n_\V} \cong (-)^n$.
\end{proof}

We can now understand coproducts, products and exponentials in $\NN_\V$:

\begin{lemma}
\label{lem:NN}
If $\V$ is any cartesian closed category with chosen coproducts of the initial object then
$\NN_\V$ is cartesian closed, with finite coproducts.  The unique initial object of $\NN_V$
is $0_\V$.  The binary coproducts in $\NN_\V$ are unique, given by
\[     m_\V + n_\V = (m + n)_\V . \]
The unique terminal object of $\NN_\V$ is $1_\V$, and the binary products are unique, given by
\[      m_\V \times n_\V = (mn)_\V  .\]
Exponentials in $\NN_\V$ are also unique, given by 
\[     {m_\V}^{n_\V} = (m^n)_\V   .\]
\end{lemma}

\begin{proof}
 In $\V$ we know that $0_\V$ is an initial object and $1_\V$ is a terminal object, by
 definition.  Since the subcategory $\NN_\V$ is skeletal $0_\V$ is the unique 
 initial object and $1_\V$ is the unique terminal object in $\NN_\V$.   Similarly, in
 $\V$ we have defined $(m+n)_\V$ to be a coproduct of $m_\V$ and $n_\V$, so 
 in $\NN_\V$ it is the unique such, and we can unambiguously write 
 \[      m_\V + n_\V = (m + n)_\V .  \]
 Products distribute over coproducts in any cartesian closed category, so in $\V$ we have 
 \[   m_\V \times n_\V \cong (1_\V + \cdots + 1_\V) \times (1_\V + \cdots + 1_\V) 
 \cong (mn)_\V \]
 where in the second step we use the distributive law twice.
 It follows that $\NN_\V$ has finite products, and since this subcategory is skeletal
 they are unique, given by
 \[   m_\V \times n_\V = (mn)_\V.  \]
 Finally, by Lemma \ref{lem:powers_1} we have
 \[  {m_\V}^{n_\V} \cong m_\V^n \cong \prod_{i = 1}^n m_\V \cong
 (m^n)_\V .\]
 It follows that $\NN_\V$ has exponentials, and since this subcategory is skeletal they
 are unique, given by
 \[     m_\V^{n_\V} = (m^n)_\V .  \qedhere\]
\end{proof}

We warn the reader that $\hom(m_\V,n_\V)$ may not have $n^m$ elements.  It does
in $\sSet,\Cat,\Pos$ and of course $\Set$, but not in $\V = \Set^k$, where
$|\hom(m_\V, n_\V)| = n^{km}$.    In fact, whenever $\NN_\V$ has finite hom-sets
it is equivalent to $\FinSet^k$ for some $k$.   The reason is that $2_\V$ is 
an internal Boolean algebra in $\V$, so its set of elements $\hom(1_\V,2_\V)$ 
must be some Boolean algebra $B$ in $\Set$.   A further argument due to Garner and Trimble shows that $\NN_\V$ is completely characterized, up to equivalence, by this Boolean algebra, and any Boolean algebra can occur \cite{nCafe}.   If this Boolean algebra is finite it must be isomorphic to $\{0,1\}^k$ for some $k \ge 0$.  In this case, $\NN_\V$ is equivalent to $\FinSet^k$.

Now suppose $\C$ is a $\V$-category.   The question arises whether the
power of an object $c \in \C$ by $n_\V$ must also be the $\V$-product of $n$ copies
of $c$.   The answer is yes:

\begin{lemma}
\label{lem:powers_2}
Let $\C$ be a $\V$-category and $c \in \Obj(\C)$.  Then the power $c^{n_\V}$ exists
if and only if the $n$-fold $\V$-product $c^n$ exists, in which case they are isomorphic.
\end{lemma}

\begin{proof}
In Section \ref{sec:enrichment} we saw that an object $b \in \Obj(\C)$ is an $n$-fold $\V$-product of copies of $c$ precisely when it is equipped with a universal cone 
\[        p \maps 1_\V \to \C(b,c)^n . \]
Similarly, $b$ is an $n_\V$-power of $c$ when it is equipped with a universal
cone 
\[       q \maps 1_\V \to \C(b,c)^{n_\V} .\]
The universality properties have the same form, and by Lemma \ref{lem:powers_1} the functors $(-)^n \maps \V\to \V$ and $(-)^{n_\V} \maps \V\to \V$ are naturally isomorphic.   Thus, given either sort of universal cone we get the other, so an object is an $n$-fold product of copies of $c$ if and only if it is the $n_\V$-power of $c$. 
\end{proof}

\begin{lemma}
\label{lem:powers_3}
Suppose $\C$ is a $\V$-category such that every object is the $n$-fold $\V$-product $c^n$ of some object $c$.   Then a $\V$-functor $F \maps \C \to \D$ preserves finite $\V$-products if and only if it preserves powers by all objects of $\NN_\V$.
\end{lemma}

\begin{proof}
Define a ``finite $\V$-power'' to be a finite $\V$-product of $n$ copies of the same object.
The $\V$-functor $F$ preserves finite $\V$-powers if and only if it maps any universal cone
\[        p \maps 1_\V \to \C(b,c)^n  \]
in $\C$ to a universal cone in $\D$.   Similarly, $F$ preserves powers by all objects of $\NN_\V$ if and only if it maps any universal cone
\[       q \maps 1_\V \to \C(b,c)^{n_\V}  \]
in $\C$ to a universal cone in $\D$.   Two kinds of universality are involved here, but
since they have the same form, and since Lemma \ref{lem:powers_1} says the functors $(-)^n \maps \V\to \V$ and $(-)^{n_\V} \maps \V\to \V$ are naturally isomorphic,
it follows that $F$ preserves finite $\V$-powers if and only if it preserves powers by all objects of $\NN_\V$.

It thus suffices to show that $F$ preserves finite $\V$-products if and only if it preserves
finite $\V$-powers.  This follows from the assumption that every object is the $n$-fold $\V$-product $c^n$ of some object $c$. 
\end{proof}

\begin{lemma}
\label{lem:powers_4}
Let $\V$ be cartesian closed with chosen finite coproducts of the terminal object and let
$\T$ be a $\V$-category.  These conditions for a $\V$-functor $\tau \maps \A_\V \to \T$ are equivalent:
\begin{enumerate}
\item $(T,\tau)$ is a $\V$-theory,
\item $\tau$ preserves finite $\V$-products,
\item $\tau$ preserves powers by objects of $\NN_\V$.
\end{enumerate}
\end{lemma}

\begin{proof} Conditions 1 and 2 are equivalent by definition.
Since $\A_\V = \underline{\NN}_\V^\op$, finite $\V$-products in $\A_\V$ are 
the same as finite $\V$-coproducts in $\underline{\NN}_\V$, which are the same as 
finite coproducts in $\NN_\V$.   Since every object in $\underline{\NN}_\V$ is a 
finite coproduct of copies of $1_\V$, Lemma \ref{lem:powers_3} implies that conditions
2 and 3 are equivalent.
\end{proof}

\begin{lemma}
\label{lem:powers_5}
Given a $\V$-theory $(\T,\tau)$ and a $\V$-functor $\mu \maps \T \to \C$,
the following conditions are equivalent:
\begin{itemize}
\item $\mu$ is a model of $(\T,\tau)$,
\item $\mu$ preserves finite $\V$-products,
\item $\mu$ preserves powers by objects of $\NN_\V$.
\end{itemize}
\end{lemma}

\begin{proof}
Conditions 1 and 2 are equivalent by definition.   Since $\tau$ is the identity on
objects and preserves $\V$-products each object of $\T$ is of the form $t^n$ 
where $t = \tau(1_\V)$.   Thus, Lemma \ref{lem:powers_3} implies that
conditions 2 and 3 are equivalent.
\end{proof}
\nocite{*}
\bibliographystyle{eptcs}
\bibliography{williams}
\end{document}